\documentclass[11pt]{article}

\usepackage{graphicx}
\usepackage{amssymb}
\usepackage{amsmath}
\usepackage{amsthm}
\usepackage{latexsym}
\usepackage[english]{babel}
\usepackage{appendix}
\usepackage{tikz}
\usetikzlibrary{matrix,arrows,decorations.pathmorphing}

\usepackage{tikz-cd}
\usepackage{hyperref}

\newtheorem{theorem}{Theorem}[section]

\newtheorem{lemma}{Lemma}[section]
\newtheorem{proposition}{Proposition}[section]
\newtheorem{remark}{Remark}[section]

\newtheorem{corollary}{Corollary}[section]

\def \C {\mathbb{C}}
\def \R {\mathbb{R}}
\def \Z {\mathbb{Z}}

\def \N {\mathbb{N}}

\begin{document}

\title{Conformally Invariant Dirac Equation with Non-Local Nonlinearity}

\author{ Ali Maalaoui$^{(1)(2)}$ \& Vittorio Martino$^{(3)}$  \& Lamine Mbarki$^{(4)}$}
\addtocounter{footnote}{1}
\footnotetext{Department of Mathematics, Clark University, 950 Main Street, Worcester, MA 01610, USA. E-mail address: 
{\tt{amaalaoui@clarku.edu}}}
\addtocounter{footnote}{1}
\footnotetext{Department of Mathematics, MIT, 77 Massachusetts Avenue
Cambridge, MA 02139-4307. E-mail address: 
{\tt{maala650@mit.edu}}}

\addtocounter{footnote}{1}
\footnotetext{Dipartimento di Matematica, Alma Mater Studiorum - Universit\`a di Bologna. E-mail address:
{\tt{vittorio.martino3@unibo.it}}}
\addtocounter{footnote}{1}
\footnotetext{Department of Mathematics, Faculty of Sciences Tunis, University of Tunis el Manar, Tunis,
 Tunisia. E-mail address:
{\tt{mbarki.lamine2016@gmail.com}; \tt{lamine.mbarki@fst.utm.tn}}}

\date{}
\maketitle

\vspace{5mm}

{\noindent\bf Abstract}  {\small We study a conformally invariant equation involving the Dirac operator and a non-linearity of convolution type. This non-linearity is inspired from the conformal Einstein-Dirac problem in dimension 4. We first investigate the compactness, bubbling and energy quantization of the associated energy functional then we characterize the ground state solutions of the problem on the standard sphere. As a consequence, we prove an Aubin-type inequality that assures the existence of solutions to our problem and in particular the conformal Einstein-Dirac problem in dimension 4. Moreover, we investigate the effect of a linear perturbation to our problem, leading us to a Brezis-Nirenberg type result.}

\vspace{5mm}


\vspace{4mm}


\vspace{4mm}


\section{Introduction and motivation}
Let $M$ be a closed (compact, without boundary) manifold of dimension $n\geq 3$, endowed with a fixed Riemannian metric $g$ and a spin structure $\Sigma_g M$. Let $D_g$ be the Dirac operator acting on spinors $\psi\in \Sigma_g M$. Let us introduce the Einstein-Dirac functional 
\begin{equation}\label{Dirac-Einstein functional} 
\mathcal{E}(g,\psi)=\int_{M}R_{g} + \langle D_{g}\psi,\psi\rangle - \lambda|\psi|^{2} \;dv_{g} 
\end{equation}
where $R_g$ is the Scalar curvature of the metric $g$ and $\lambda$ is a real parameter. Critical points of $\mathcal{E}$ are solutions of the Einstein-Dirac equations (see for instance \cite{MaalaouiMartino2022})
\begin{equation}\label{Dirac-Einstein equations}
\left\{\begin{array}{ll}
Ric_{g}-\displaystyle\frac{R_{g}}{2}g=T_{g,\psi}\\
\\
D_{g}\psi=\lambda \psi
\end{array}
\right.
\end{equation}
where $Ric_{g}$ is the Ricci tensor and $T_g$ is the stress–energy tensor given by
$$T_{g,\psi}(X,Y)=\displaystyle -\frac{1}{4}\langle X\cdot\nabla_Y \psi + Y\cdot\nabla_X \psi, \psi \rangle, \quad X,Y \in TM \; ,$$
here $\cdot$ and $\nabla$ denote the Clifford multiplication and the connection on  $\Sigma_g M$ (see \cite{KimFriedrich2000}, \cite{Friedrich2000}). This functional was investigated in dimensions 3 and 4 in \cite{MaalaouiMartino2022}, where the authors study the limits of such structures under natural bounds on the diameter and the curvature. We also mention that the first equation in $(\ref{Dirac-Einstein equations})$ is similar in structure to the semi-classical gravity model coupling gravity with matter in a way that only the matter fields are quantified, we refer to \cite{BGDB} for more details about the model. Now, if we restrict the variations of the metric to a given conformal class, that is $\tilde{g}=u^{\frac{4}{n-2}}g$ and $\tilde{\psi}=u^{\frac{1-n}{n-2}}\psi \in \Sigma_{\tilde{g}}M$, we obtain the following functional

\begin{equation}\label{Dirac-Einstein functional conformal}
\mathcal{E}\big( \widetilde{g},\widetilde{\psi} \big)=\int_{M} u L_{g}u + \langle D_{g}\psi, \psi \rangle - \lambda u^{\frac{2}{n-2}} |\psi|^{2} \;dv_{g}=:\mathcal{E}_g(u, \psi) \;,
\end{equation}  
where $L_{g}$, here is the conformal Laplacian. The critical points of this functional solve the conformal Einstein-Dirac equations
\begin{equation}\label{DE4}
\left\{\begin{array}{ll}
L_{g} u=\displaystyle\frac{\lambda}{n-2}|\psi|^{2} u^{\frac{4-n}{n-2}}\\
\\
D_{g}\psi=\lambda u^{\frac{2}{n-2}} \psi
\end{array}.
\right.
\end{equation}
The case $n=3$ was investigated in \cite{BorrelliMaalaouiMartino2023}, \cite{GuidiMaalaouiMartino2021}, \cite{MaalaouiMartino2019} and the case $n=2$ corresponds to the super-Liouville problem investigated in \cite{JMW1,JMW2, JWZZ}. We also mention the recent work of Sire and Xu \cite{YX} where the authors adopt a flow approach to investigate the problem. We notice that in dimension $n=4$, the system $(\ref{Dirac-Einstein functional conformal})$ takes a more approachable structure. That is, one can solve the first equation, finding $u$ in terms of $|\psi|^{2}$ by using the Green's function of the conformal Laplacian, then inserting it in the second equation one has a single equation that can be written as
$$D_{g}\psi=\Big(\int_{M}G(x,y)|\psi|^{2}(y)\ dv_{g}(y)\Big) \psi,$$
where $G$ is the Green's function of the conformal Laplacian $L_{g}$. Due to the singularity of the Green's function ($G(x,y)\sim \frac{1}{|x-y|^{2}}$, when $x$ is close to $y$), one can see the similarities with other classical equations in the literature. It is in fact surprising how this type of equations appears naturally in different models in physics. For instance, based on the work in \cite{GG}, the Schrodinger-Newton model can be derived from the Einstein-Dirac model through a non-relativistic limit and we recall here that the Schrodinger-Newton equation in $\R^3$ takes the form
$$i\frac{\partial \psi }{\partial t}=-\Delta \psi -c\Big(\int \frac{|\psi(t,y)|^{2}}{|x-y|}\ dy \Big) \psi,$$
where we see clearly the convolution term that appears in the non-linearity. Notice that the static solutions correspond to a version of the Choquard equations. Hence, our problem can be seen as a spinorial version of the Choquard equation, we refer the reader to the survey \cite{MS} and the references therein. But also, this equation is similar to the semi-classical Hartree's equation and the Lieb-Yau conjecture for the pseudo-relativistic Boson stars model \cite{Lieb1,Lieb2,Enno}. Another important model where such equation appears is the Dirac-Maxwell system studied in \cite{DWX} (see also the references therein).

\noindent
In this work, we propose to study a general problem with the same structure and conformal invariance properties that would capture the solutions to the conformal Einstein-Dirac equation in dimension 4. We then consider the following equation
\begin{equation}\label{eq}
D_g\psi=(G_g^{s}*|\psi|^{2})\psi,
\end{equation}
and its linear perturbation
\begin{equation}\label{eq2}
D_{g}\psi=\lambda \psi +(G_g^{s}*|\psi|^{2})\psi,
\end{equation}
where we denoted by
$$(G_g^{s}*f)(x):=\int_{M}G_g^{s}(x,y)f(y) \; dv_{g}(y) \; $$
the convolution of a given function $f$ with the Green's function $G_g^{s}$ of the conformal fractional Laplacian  $P_{g}^{s}$, of order $2s=n-2$.\\
These equations have a variational structure and the corresponding energy functional for $(\ref{eq2})$ is given by 
\begin{align}\label{energyfunctional}
J_{g,\lambda}(\psi) & = \frac{1}{2}\int_{M}\langle D_g\psi,\psi\rangle -\lambda |\psi|^{2}\ dv_{g}-\frac{1}{4}\int_{M} (G_g^{s}*|\psi|^{2})|\psi|^{2}\; dv_{g} \notag \\
& = \frac{1}{2}\int_{M}\langle D_g\psi,\psi\rangle -\lambda |\psi|^{2}\ dv_{g}-\frac{1}{4}\int_{M\times M} G_g^{s}(x,y)|\psi(y)|^{2} \; |\psi(x)|^{2}\; dv_{g}(y)\; dv_{g}(x) \;,
\end{align}
where $\langle \cdot ,\cdot\rangle$ is the canonical Hermitian metric defined on  $\Sigma_g M$. Notice that $J_{g,0}=:J_{g}$ is the energy functional corresponding to $(\ref{eq})$.\\
We see that, the particular choice of the parameter $s$ makes the functional $J_{g}$ invariant under a conformal change of the metric; in order to see this, for any $s$ for which the conformal fractional Laplacian is defined (see \cite{Mar, Chang-Gonzalez}), let us consider a conformal change of the metric
\begin{equation}\label{changemetric}
  \widetilde{g} = u^{\frac{4}{n-2s}}g,  \; 0<u\in C^{\infty}(M) \; .
\end{equation}
Given a spinor $\psi \in \Sigma_{g} M$, we set 
$$\widetilde{\psi} = u^{\frac{1-n}{n-2s}} \psi \in \Sigma_{\widetilde{g}} M \; ,$$
where we implicitly understand the action of a canonical isometric isomorphism between the spinor bundles $\Sigma_{\widetilde{g}} M$ and $\Sigma_{g} M$ (see \cite{KimFriedrich2000}, Section 2). In this way, we have the conformal change of the Dirac operator
$$
D_{ \widetilde{g}}\widetilde{\psi}  = u^{-\frac{n+1}{n-2s}} D_{g} \psi \; .
$$
Also, by using the conformal covariance property of the fractional Laplacian
$$P_{\widetilde{g}}^{s}(f)= u^{-\frac{n+2s}{n-2s}} P_{g}^{s}(uf) \; ,$$
we obtain the  conformal change of its Green's function
$$
G_{\widetilde{g}}^s(x,y)=u(x)^{-1} u(y)^{-1}G_{g}^s(x,y) \; .
$$
Now, taking into account the change of the volume
$$dv_{\widetilde{g}}  = u^{\frac{2n}{n-2s}}  dv_{g} \; ,$$
we substitute in (\ref{eq}) and find
\begin{align*}
J_{\widetilde{g}}(\widetilde{\psi}) & = \frac{1}{2}\int_{M}\langle D_{\widetilde{g}}\widetilde{\psi},\widetilde{\psi}\rangle \ dv_{\widetilde{g}}-\frac{1}{4}\int_{M\times M} G_{\widetilde{g}}^{s}(x,y)|\widetilde{\psi}(y)|^{2} \; |\widetilde{\psi}(x)|^{2}\; dv_{\widetilde{g}}(y)\; dv_{\widetilde{g}}(x) \;,\\
& = \frac{1}{2}\int_{M}\langle D_{g}\psi,\psi\rangle \ dv_{g}-\frac{1}{4}\int_{M\times M} G_g^{s}(x,y)|\psi(y)|^{2} \; |\psi(x)|^{2} \;  u^{\frac{4+4s-2n}{n-2s}} \; dv_{g}(y)\; dv_{g}(x) \; .
\end{align*}
Therefore, if $2s=n-2$ we obtain $J_{\widetilde{g}}(\widetilde{\psi})=J_{g}(\psi)$. In particular, this says that equation (\ref{eq}) is critical, in the sense of the conformal analysis.

\noindent
This manuscript is mainly split in two parts. In the first part, we investigate the lack of compactness of the problem, due to the conformal invariance and we prove the following bubbling and energy quantization result for the functional $J_{g}$.
\begin{theorem}\label{thmbubble} 
Let us assume that $(M,[g])$ has a positive Yamabe constant $Y_s(M,[g])$ and let $(\psi_{k})_{k\in \N}$ be a Palais-Smale sequence for $J_{g}$ at level $c\geq 0$, where $Y_s(M,[g])$ is the s-Yamabe constant, which we define in the next section.. Then there exist $\psi_{\infty}\in C^{\infty}(M,\Sigma_{g} M)$, a solution of $(\ref{eq})$, $m$ sequences of points $x_{k}^{1},\cdots, x_{k}^{m} \in M$ such that $\lim_{k\to \infty}x_{k}^{j}= x^{j}\in M$, for $j=1,\dots,m$ and $m$ sequences of real numbers $R_{k}^{1},\cdots, R_{k}^{m}$ converging to zero, such that:
\begin{itemize}
\item[ii)]  $\displaystyle \psi_{k}=\psi_{\infty}+\sum_{j=1}^{m}\phi_{k}^{j}+o(1)$ in $H^{\frac{1}{2}}(\Sigma M)$,
\item[iii)] $\displaystyle J_{g}(\psi_{k})=J_{g}(\psi_{\infty})+ 
    \sum_{j=1}^{m}J_{g_{\mathbb{R}^{n}}}(\Psi_\infty^{j})+o(1)$,
\end{itemize}
where
$$\phi_{k}^{j}=(R_{k}^{j})^{-1}\beta_{j}\sigma_{k,j}^{*}(\Psi_\infty^{j}) ,$$
with $\sigma_{k,j}=(\rho_{k,j})^{-1}$ and $\rho_{k,j}(\cdot)=exp_{x_{k}^j}(R_{k}^j \cdot)$ is the exponential map defined in a suitable neighborhood of $\R^{n}$.
Also, here $\beta_{j}$ is a smooth compactly supported function, such that $\beta_{j}=1$ on $B_{1}(x^{j})$ and $supp(\beta_{j})\subset B_{2}(x^{j})$ and $\Psi_\infty^{j}$ is the solution to our equations (\ref{eq}) on $\mathbb{R}^n$ with its Euclidian metric $g_{\R^n}$.
\end{theorem}

\noindent
As we will see in the proof, the same result holds for the functional $J_{g,\lambda}$, with the same bubbles at infinity. We also characterize the ground state solutions that appear in the bubbling phenomena in the theorem above.
\begin{theorem}\label{leasten}
Let $\psi\in C^{\infty}(\Sigma_{g_{0}}S^{n})$ be a non-trivial solution of 
\begin{equation}\label{sol sphere}
D_{g_{0}}\psi=\Big(G^{s}_{g_{0}}\ast \psi\Big)\psi, \text{ on } S^{n},
\end{equation}
where $g_{0}$ is the round metric on $S^{n}$. Then,
\begin{equation}\label{eq:low bound sol sphere}
J_{g_0}(\psi)\geq \overline{Y}(S^{n},[g_{0}]):=\frac{\lambda^{+}(S^{n}, [g_{0}])^{2} Y_s(S^{n}, [g_{0}])}{4} .
\end{equation}
Moreover, if $J_{g_{0}}(\psi)=\overline{Y}(S^{n},[g_{0}])$  then, up to a conformal change, $\psi$ is a $-\frac{1}{2}$-Killing spinor. That is, there exists a $-\frac{1}{2}$-Killing spinor $\Psi\in \Sigma_{g_{0}}S^{n}$ and a conformal diffeomorphism $f\in Conf(S^{n},g_{0})$ such that 
$$\psi=\Big(det(df)\Big)^{\frac{n-1}{2n}}F_{f^{*}g_{0},g_{0}}\Big(f^{*}\Psi\Big).$$
\end{theorem}

\noindent
As a corollary of this Theorem we have an Aubin-type inequality for the problem $(\ref{eq})$:
\begin{corollary}\label{cor1}
Under the assumptions of Theorem \ref{thmbubble}, there exists a conformally invariant constant $\overline{Y}(M,[g])>0$ with the following properties:
\begin{itemize}
\item[i)] $\overline{Y}(M,[g])\leq \overline{Y}(S^{n},[g_{0}])=\frac{\lambda^{+}(S^{n}, [g_{0}])^{2} Y_{s}(S^{n}, [g_{0}])}{4}$.
\item[ii)] If $\overline{Y}(M,[g])< \overline{Y}(S^{n},[g_{0}])$ then the problem $(\ref{eq})$ has a non-trivial solution.

\end{itemize}
\end{corollary}

\noindent
Notice that in particular, when $n=4$, we can state $ii)$ in the setting of the conformal Einstein-Dirac equation. That is, if $\overline{Y}(M,[g])< \overline{Y}(S^{n},[g_{0}])$ the conformal Einstein-Dirac problem $(\ref{DE4})$ is solvable.

\noindent
The second part of this paper deals with the existence of solutions for the linearly perturbed problem $(\ref{eq2})$. Namely, we prove a Brezis-Nirenberg type result associated to the original problem $(\ref{eq})$.
\begin{theorem}\label{thmlambda}
Assume that $(M,[g])$ has a positive Yamabe invariant and $Y_{s}(M,[g])>0$. Then for any $\lambda \not \in Spec(D_{g})$ and $\lambda>0$, there exists a non-trivial ground-state solution $\psi_{\lambda}$ for $(\ref{eq2})$. Moreover, if $\lambda \in (\lambda_{k},\lambda_{k+1})$, then $\psi_{\lambda}\to 0$ as $\lambda \to \lambda_{k+1}$.
\end{theorem}

\vspace{5mm}

{\noindent\bf Acknowledgment}

\noindent
The first author is supported by the AMS-Simons Research Enhancement Grant for PUI faculty under the project "Conformally Invariant
Non-Local Equations on Spin Manifolds". He also wants to express his gratitude to the department of Mathematics at MIT for the warm hospitality during the finalization of this manuscript.

\section{Preliminaries}
A spin structure on a riemannian manifold $(M,g)$ is a pair $(P_{Spin}(M,g),\sigma)$, where $P_{Spin}(M,g)$ is a $Spin(n)$-principal bundle and $\sigma : P_{Spin}(M,g)\to P_{SO}(M,g)$ is a 2-fold covering map, which restricts to a non-trivial covering $\kappa: Spin(n)\to SO(n)$ on each fiber. That is, the quotient of each fiber by $\Z_{2}$ is isomorphic to the frame bundle of $M$ and hence, the following diagram commutes:
\begin{center}

\tikzset{every picture/.style={line width=0.75pt}} 

\begin{tikzpicture}[x=0.75pt,y=0.75pt,yscale=-1,xscale=1]

\draw    (173,59) -- (380,59) ;
\draw [shift={(382,59)}, rotate = 180] [color={rgb, 255:red, 0; green, 0; blue, 0 }  ][line width=0.75]    (10.93,-3.29) .. controls (6.95,-1.4) and (3.31,-0.3) .. (0,0) .. controls (3.31,0.3) and (6.95,1.4) .. (10.93,3.29)   ;
\draw    (150,80) -- (247.36,147.86) ;
\draw [shift={(249,149)}, rotate = 214.88] [color={rgb, 255:red, 0; green, 0; blue, 0 }  ][line width=0.75]    (10.93,-3.29) .. controls (6.95,-1.4) and (3.31,-0.3) .. (0,0) .. controls (3.31,0.3) and (6.95,1.4) .. (10.93,3.29)   ;
\draw    (415,79) -- (314.63,149.85) ;
\draw [shift={(313,151)}, rotate = 324.78] [color={rgb, 255:red, 0; green, 0; blue, 0 }  ][line width=0.75]    (10.93,-3.29) .. controls (6.95,-1.4) and (3.31,-0.3) .. (0,0) .. controls (3.31,0.3) and (6.95,1.4) .. (10.93,3.29)   ;

\draw (86,43) node [anchor=north west][inner sep=0.75pt]    {$P_{Spin}( M,g)$};
\draw (388,43) node [anchor=north west][inner sep=0.75pt]    {$P_{SO}( M,g)$};
\draw (257,151) node [anchor=north west][inner sep=0.75pt]    {$( M,g)$};
\draw (275,30) node [anchor=north west][inner sep=0.75pt]    {$\sigma $};

\end{tikzpicture}
\end{center}

\noindent
We denote by $\mathbb{S}_{n}$ the unique (up to isomorphism) irreducible complex $Cl_{n}$-module such that $Cl_{n}\otimes \C \equiv End_{\C}(\mathbb{S}_{n})$ as a $\C$-algebra, where $Cl_{n}$ denotes the Clifford algebra of $\R^{n}$. This allows us to define the spinor bundle $\Sigma_{g} M$ as
$$\Sigma_{g} M:=P_{Spin}(M,g)\times_{\sigma} \mathbb{S}_{n}.$$
In fact, $\Sigma_{g} M$ is a Hermitian bundle equipped with a metric connection induced by the Levi-Civita connection on $TM$, that we will denote by $\nabla$. Moreover, there is a natural Clifford multiplication defined by the action of $TM$ on $\Sigma_{g} M$. We can summarize the main properties of the spinor bundle in the following few points:

\begin{itemize}
\item For all $X, Y \in C^{\infty}(M,TM)$ and $\psi \in C^{\infty}(M,\Sigma_{g} M)$ we have $X\cdot Y \cdot \psi +Y\cdot X \cdot \psi=-2g(X,Y)\psi$. Here, "$\cdot$" denotes the Clifford multiplication.
\item If $(\cdot,\cdot)$ denotes the Hermitian metric on $\Sigma_{g} M$, then for all $X \in C^{\infty}(M,TM)$ and $\psi, \phi \in C^{\infty}(M,\Sigma_{g} M)$ we have $(X\cdot \psi , \phi)=-(\psi, X\cdot \phi)$.
\item For all $\psi, \phi \in C^{\infty}(M,\Sigma_{g} M)$ and $X\in C^{\infty}(M,TM)$, then $X(\psi,\phi)=(\nabla_{X}\psi,\phi)+(\psi,\nabla_{X} \phi)$.
\item For all $X, Y \in C^{\infty}(M,TM)$ and $\psi \in C^{\infty}(M,\Sigma_{g} M)$ we have $\nabla_{X} (Y\cdot \psi)=(\nabla_{X}Y)\cdot \psi+Y\cdot \nabla_{X}\psi$.
\end{itemize}

\noindent
For the rest of the paper, we let $\langle \cdot, \cdot \rangle:= Re (\cdot,\cdot)$. Then $\langle \cdot, \cdot \rangle$ defines a metric on $\Sigma_{g} M$.
The Dirac operator $D_{g}$ is then defined on $C^{\infty}(M,\Sigma_{g} M)$ as the composition of the Clifford multiplication and the connection $\nabla$. Indeed, if $(e_{1}, \cdots, e_{n})$ is a local orthonormal frame around a point $p\in M$ and $\psi \in C^{\infty}(M,\Sigma_{g} M)$ then one can locally define $D_{g}$ by  
$$D_{g}\psi :=\sum_{i=1}^{n}e_{i}\cdot \nabla_{e_{i}} \psi.$$
The Dirac operator is a natural first order operator acting on smooth sections of $\Sigma_{g} M$. Moreover, if $M$ is compact, then $D_{g}$ is essentially self-adjoint on $L^{2}(\Sigma_{g} M):=L^{2}(M,\Sigma_{g} M)$, with compact resolvent. In particular, there exists a complete orthonormal basis $(\varphi_{k})_{k\in \Z}$ of $L^{2}(\Sigma_{g} M)$ consisting of eigenspinors of $D_{g}$. That is $D_{g}\varphi_{k}=\lambda_{k}\varphi_{k}$, with $\lambda_{k}\to \pm \infty$ when $k\to \pm \infty$. We will use the convention that $\lambda_{k}>0$ (resp. $\lambda_{k}<0$) when $k>0$ (resp. $k<0)$.

\begin{proposition}[\cite{Bour, Friedrich2000}]
Consider a compact spin manifold $(M,g,\Sigma_{g}M)$, then
\begin{itemize}
\item[i)] The Dirac operator $D_{g}$ is conformally invariant. That is, if $\hat{g}:=e^{2u}g$, then there exists a unitary isomorphism $F_{g,\hat{g}}:\Sigma_{g}M \to \Sigma_{\hat{g}}M$ so that for $\varphi \in C^{\infty}(M,\Sigma_{g}M)$, $$D_{\hat{g}}(e^{-\frac{n-1}{2}u}F_{g,\hat{g}}(\varphi))=e^{-\frac{n+1}{2}u}F_{g,\hat{g}}(D_{g}\varphi).$$
\item[ii)] For $\varphi \in C^{\infty}(M,\Sigma_{g}M)$, $D_{g}^{2}\varphi =-\Delta_{g}\varphi +\frac{R_{g}}{4}\varphi$, where $R_{g}$ is the scalar curvature. 
\end{itemize}
\end{proposition}

\noindent
In what follows, we will identify spinors $\varphi \in \Sigma_{g}M$ with their isomorphic image $F_{g,\hat{g}}(\varphi)$, unless there is a specific distinction. Notice that as a result of the two points of the previous Proposition we have that $D_{g}$ is invertible if the Yamabe invariant of $(M,g)$ is positive.\\
We can define now the (unbounded) operator $|D_{g}|^{s}:L^{2}(\Sigma_{g}M)\to L^{2}(\Sigma_{g}M)$, for $s>0$ by $$|D_{g}|^{s}\psi =\sum_{k\in \Z} |\lambda_{k}|^{s}a_{k}\varphi_{k},$$
for $\psi=\sum_{k\in \Z}a_{k}\varphi_{k}$. The Sobolev space $H^{\frac{1}{2}}(\Sigma_{g} M)$ is then defined by
$$H^{\frac{1}{2}}(\Sigma_{g} M):=\{\psi \in L^{2}(\Sigma_{g} M); |D_{g}|^{\frac{1}{2}}\psi \in L^{2}(\Sigma_{g} M)\}.$$
This function space is equivalent to the classical $H^{\frac{1}{2}}$-Sobolev space and will be endowed with the inner product $\langle \cdot, \cdot \rangle_{\frac{1}{2}}$ defined by
$$\langle \psi,\phi \rangle_{\frac{1}{2}}:=\int_{M}\langle |D_{g}|^{\frac{1}{2}}\psi,|D_{g}|^{\frac{1}{2}}\phi \rangle \ dv_{g}, \forall \psi, \phi \in H^{\frac{1}{2}}(\Sigma_{g} M).$$
Notice that this inner product defines a natural semi-norm on $H^{\frac{1}{2}}(\Sigma_{g} M)$ by setting
$$\|\psi \|_{\frac{1}{2}}:=\||D_{g}|^{\frac{1}{2}}\psi\|_{L^{2}}.$$
This semi-norm becomes a norm when $D_{g}$ is invertible. Using the spectral resolution of $D_{g}$, we can split the space $H^{\frac{1}{2}}(\Sigma_{g} M)$ in a convenient way that fits our analysis. That is, we can write 
\begin{equation}\label{split}
H^{\frac{1}{2}}(\Sigma_{g} M)=H^{\frac{1}{2},-}\oplus H^{\frac{1}{2},0}\oplus H^{\frac{1}{2},+} \; ,
\end{equation}
with
$$H^{\frac{1}{2},-}:=\overline{\text{span}\{\varphi_i\}_{i<0}},\quad
H^{\frac{1}{2},0}:=\ker D_{g}, \quad
H^{\frac{1}{2},+}:=\overline{\text{span}\{\varphi_i\}_{i>0}}.$$
This leads to the natural projectors $P^{\pm}:H^{\frac{1}{2}}(\Sigma_{g}M) \to H^{\frac{1}{2},\pm}$ and for $\psi \in H^{\frac{1}{2}}(\Sigma_{g}M)$ we will write $\psi^{+}:=P^{+}\psi$ and $\psi^{-}:=P^{-}\psi$. 
Since we will be considering a linear perturbation of the Dirac operator, we introduce the following operator $D_{\lambda}:=D_{g}-\lambda$, for $\lambda \not \in Spec (D_{g})$. Notice that $D_{\lambda}$ has a similar spectral decomposition and hence we can introduce a similar adapted splitting as in $(\ref{split})$, for the space $H^{\frac{1}{2}}$. That is:
$$H^{\frac{1}{2}}(\Sigma_{g} M)=H^{-}_{\lambda}\oplus H^{+}_{\lambda} \; ,$$
The new adapted inner product and norm are then defined by
$$\langle \psi, \phi\rangle _{\lambda}=\int_{M}\langle |D_{\lambda}|^{\frac{1}{2}}\psi, |D_{\lambda}|^{\frac{1}{2}}\phi\rangle\ dv_{g} \quad \text{ and } \quad \|\psi\|_{\lambda}=\||D_{\lambda}|^{\frac{1}{2}}\psi\|_{L^{2}}, \forall \psi, \phi \in H^{\frac{1}{2}}(\Sigma_{g}M).$$

\noindent
We recall now some of the properties of the conformal fractional Laplacian and GJMS operators. A good reference for the material discussed in this paragraph is \cite{Mar}. For this purpose, we consider a Poincar\'{e}-Einstein manifold $(X,g^{+})$ with conformal infinity $(M, [g])$. Therefore, there exists a geodesic defining function $\rho$ such that in a neighborhood of $M$ in $X$ of the form $M\times (0,\varepsilon)$, the metric $g^{+}$ takes the form $$g^{+}=\frac{1}{\rho^{2}}(d\rho^{2}+g_{\rho}),$$
where $g_{\rho}$ is a one parameter family of metrics on $M$ such that $g_{0}=g$. Moreover, we have $Ric_{g^{+}}=-ng_{+}$. In fact, one can weaken this last Einstein equality to be up to a term of the form $O(\rho^{n-2})$ if $n$ is even and up to a term of $O(\rho^{\infty})$, if $n$ is odd. One then can solve (even formally) the following generalized eigenvalue problem: for $s\in (0,\frac{n}{2})$ and $s\not \in \N$, and $u\in C^{\infty}(M)$
$$\left\{\begin{array}{ll}
-\Delta_{g^{+}}U-(\frac{n}{2}+s)(\frac{n}{2}-s)U=0 \text{ in } X\\
\\
U=\rho^{\frac{n}{2}-s}F(\rho)+\rho^{\frac{n}{2}+s}G(\rho);
\end{array}
\right.$$
where $F,G \in C^{\infty}(\overline{X},\overline{g}:=\rho^{2}g^{+})$ and $F_{|\rho=0}=u$. For the details about the construction of such solution, we refer the reader to \cite{GZ}. The operator $S(s): u\mapsto G_{|\rho=0}$ is called the scattering operator. The fractional conformal Laplacian is then defined by
$$P_{g}^{s}u:=d_{s}S(s)u, \quad \text{ where } d_{s}=2^{2s}\frac{\Gamma(s)}{\Gamma(-s)}.$$
The properties of the operator $P_{g}^{s}$ can be summarized as follows:
\begin{proposition}[\cite{Chang-Gonzalez,Mar,GZ}]
Using the definition above, we have:
\begin{itemize}
\item $P_{g}^{s}$ is a self-adjoint pseudo-differential operator on $M$ with principal symbol coinciding with the one of $(-\Delta_{g})^{s}$. 
\item $P_{g}^{s}$ is a conformally covariant operator. That is, if $\hat{g}=u^{\frac{4}{n-2s}}g$ then 
$$P_{\hat{g}}^{s}(\cdot)=u^{-\frac{n+2s}{n-2s}}P_{g}^{s}(u\cdot) .$$
\item $(P_{g}^{s})_{s\in (0,\frac{n}{2})}$  constitutes a meromorphic family of operators that has potential simple poles when $s\in \N$. These poles are compensated by the normalization constant $d_{s}$, making the family then holomorphic.
\item When $M$ is the Euclidean space $\R^{n}$, we have $P_{g_{\R^{n}}}^{s}=(-\Delta_{\R^{n}})^{s}$.
\end{itemize}
\end{proposition}

\noindent
Notice that $P_{g}^{s}$ is a non-local operator when $s\not \in \N$. But when $s=k$ is an integer, then $P_{g}^{k}$ is a differential operator and it coincides with the classical GJMS operators \cite{GJMS}. In fact, one can check that
$$P_{g}^{1}=L_{g}:= -\Delta_{g}+\frac{n-2}{4(n-1)}R_{g},$$
and
$$P_{g}^{2}=-\Delta_{g}^{2}+div(a_{n}R_{g}g+b_{n}Ric_{g})d+\frac{n-4}{2}Q_{2},$$
where $a_{n}$ and $b_{n}$ are two constants depending on $n$ and $Q_{2}$ is, up to a multiplicative constant, the classical Q-curvature. In a similar way, one can define the fractional Q-curvature by:
$$Q_{g}^{s}:=\frac{P_{g}^{s}(1)}{(n-2s)}.$$
For example, $Q_{g}^{1}=\frac{R_{g}}{4(n-1)}$. We will restrict ourselves to the case $0<2s<n$. One now can formulate the fractional Yamabe problem, which addresses the question of prescribing constant $Q_{g}^{s}$-curvature. This is equivalent to solving the problem of finding $u>0$ such that
\begin{equation}\label{yams}
P_{g}^{s}u=cu^{\frac{n+2s}{n-2s}}.
\end{equation}
As in the classical Yamabe problem, the sign of the constant $c$ is a conformal invariant and it is determined by the sign of  $\int_{M}Q_{g}^{s}\ dv_{g}$. We will focus on the positive case, that is, when $\int_{M}Q_{g}^{s}\ dv_{g}>0$. We consider then the functional $I_{s}:[g]\to \R$ defined by
$$I_{s}(h):=\frac{\int_{M}Q_{h}^{s} \ dv_{h}}{\Big(\int_{M}dv_{h}\Big)^{\frac{n-2s}{n}}}.$$
Taking $h:=u^{\frac{4}{n-2s}}g$ yields
$$I_{s}(u,g):=I_{s}(h)=\frac{\int_{M}uP_{g}^{s}u\ dv_{g}}{\Big(\int_{M}u^{\frac{2n}{n-2s}}\ dv_{g}\Big)^{\frac{n-2s}{n}}}.$$
Therefore, finding a critical point of $I_{s}$ is equivalent to finding a solution to $(\ref{yams})$. We can then define the $s$-Yamabe constant by
\begin{equation}\label{yamc}
Y_{s}(M,[g]):=\inf\{I_{s}(h); h\in [g]\}=\inf\{ I_{s}(u,g); u>0 \text{ and } u\in H^{s}(M)\}.
\end{equation}
Notice that when $Y_{s}(M,[g])>0$ (as in the case of $(S^{n},[g_{0}])$, one can define an equivalent $H^{s}$-norm on $M$ by setting 
$$\|u\|_{H^{s}}:=\Big(\int_{M}uP_{g}^{s}u\ dv_{g}\Big)^{\frac{1}{2}}.$$
In this case, the best constant in the Sobolev embedding $H^{s}(M)\hookrightarrow L^{\frac{2n}{n-2s}}(M)$ coincides with $Y_{s}(M,[g])^{-\frac{1}{2}}$. We will assume from now on that the Green's function of $P_{g}^{s}$ is positive. This is not a necessary condition but it does make the notations in the proofs easier. In fact, there are several conformally invariant assumptions that we can consider if we truly need the positivity of the $G_{g}^{s}$. We refer the reader to \cite{CC} where the authors address the positivity of the Green's function in certain ranges of the parameter $s$.
We point out that when $Y_{s}(M,[g])>0$ and $G^{s}_{g}$ is the Green's function of $P_{g}^{s}$, then for any $f\in C^{\infty}(M)$, we have
$$\int_{M\times M}G_{g}^{s}(x,y)f(x)f(y)\ dv_{g}(y)dv_{g}(x)\geq 0.$$
We summarize here some of the useful properties of the $G^{s}_{g}$ that we will be using in the next sections.
\begin{proposition}\label{propgreen}
We consider a compact Riemannian manifold $(M,g)$ as above and fix $ 0<s<\frac{n}{2}$. Assume that $Y_{s}(M,[g])>0$, then the Green's function of $P_{g}^{s}$ satisfies:
\begin{itemize}
\item[i)] $G_{g}^{s}$ is continuous and bounded away from the diagonal $\triangle_{M\times M}:=\{(x,x)\in M\times M\}$.
\item[ii)] For $p_{0}\in M$ there exists a small neighborhood $U_{p_{0}}$ around $p_{0}$ in $M$ such that in normal coordinates around $p_{0}$,
$$G_{g}^{s}(x,y)=G_{g_{\R^{n}}}^{s}(x,y)+r(x,y), \forall x,y\in U_{p_{0}},$$
where $G_{g_{\R^{n}}}^{s}(x,y)=\frac{c_{n,s}}{|x-y|^{n-2s}}$ is the Green's function of $(-\Delta_{\R^{n}})^{s}$ and there exists $C>0$ such that
$$|r(x,y)|\leq \frac{C}{|x-y|^{n-2s-1}}, \forall x,y\in U_{p_{0}}.$$
\end{itemize}
\end{proposition}

\section{Regularity}

\noindent
In this section, we will focus on the study of regularity of solutions of $(\ref{eq})$, actually the same results hold for $(\ref{eq2})$. In the sequel, for the sake of simplicity, we will omit the dependence on the metric; for instance $\Sigma M = \Sigma_g M$ and so on. For the same reason, we will denote the functional spaces depending only $M$; for instance $L^p(M)= L^p(M,\Sigma M)$. The $H^{\frac{1}{2}}$-norm will also be denoted simply by $\|\cdot\|$.  \\
Our objective here is to show that weak solutions of $(\ref{eq})$ are indeed classical solutions. First of all, we consider the Sobolev space $H^{\frac{1}{2}}(M)$ as defined in the previous section. Here we just recall that there exists a continuous Sobolev embedding
$$H^{\frac{1}{2}}( M) \hookrightarrow L^p( M), \quad 1\leq p \leq  \frac{2n}{n-1},$$
this is also compact if $1\leq p <\frac{2n}{n-1}$.

\noindent
We will say that $\psi \in L^{\frac{2n}{n-1}}(M)$ is a weak solution of $(\ref{eq})$ if
$$ \int_M \langle D\phi,\psi\rangle \ dv =  \int_M (G^{s}*|\psi|^{2}) \langle \phi,\psi\rangle \ dv,$$
for all $\phi \in C^{\infty}(M)$. Notice that for a fixed metric $g$, the critical points of $J_{g}$ are weak solutions of $(\ref{eq})$. We then have the following result: 

\begin{theorem}\label{reg}
Let $\psi \in L^{\frac{2n}{n-1}}(M)$ be a weak solution of $(\ref{eq})$. Then $\psi \in C^{\infty}(M)$.
\end{theorem}

\noindent
The idea of the proof is somehow similar to that in \cite{I} (see also \cite{MaalaouiMartino2019}), but we will provide here the full details since the non-linearity, in this case, is non-local.
\begin{proof}
Given a small $r>0$, we consider two cut-off functions $\eta_{1}$ and $\eta_{2}$ such that $\eta_{1}$ is supported in $B_{3r}$ and equals $1$ on $B_{2r}$. Similarly, $\eta_{2}=1$ on $B_{\frac{r}{2}}$ and supported in $B_{r}$. Now, one has
\begin{equation}\label{eqreg}
D(\eta_{2}\psi)=(G^{s}*|\psi|^{2})\eta_{2}\psi +\nabla\eta_{2}\cdot \psi.
\end{equation}
On the other hand, we will write
\begin{equation}\label{eqreg2}
G^{s}*|\psi|^{2}=G^{s}*(\eta_{1}|\psi|^{2} + (1-\eta_{1})|\psi|^{2})=u_{1}+u_{2},
\end{equation}
so that
$$D(\eta_{2}\psi)=u_{1}\eta_{2}\psi+\eta_{2}u_{2}\psi+\nabla\eta_{2}\cdot \psi.$$
Now, for $1\leq p<n$, let $P: W^{1,p}(M)\to L^{p}(M)$ defined by
$$Pv=u_{1}v.$$
We notice that
\begin{align}
\|u_{1}v\|_{L^{p}}&\leq \|u_{1}\|_{L^{n}}\|v\|_{L^{\frac{np}{n-p}}}\notag\\
&\leq C\|\psi\|_{L^{\frac{2n}{n-1}}(B_{3r})}^{2}\|v\|_{W^{1,p}(M)}.
\end{align}
Thus, we have
$$\|P\|_{Op}\leq C\|\psi\|_{L^{\frac{2n}{n-1}}(B_{3r})}^{2} \; ,$$
where $\|\cdot\|_{Op}$ stands for the operator norm. Since $D:W^{1,p}(M)\to L^{p}(M)$ is invertible, we have for $r$ small enough, that $D-P:W^{1,p}(M)\to L^{p}(M)$ is invertible. Noticing that $\nabla \eta_{2} \psi+\eta_{2}u_{2}\psi \in L^{\frac{2n}{n-1}}(M)$, there exists a unique solution $v_{0}\in W^{1,p}(M)$ of
$$Dv_{0}=u_{1}v_{0}+\nabla \eta_{2} \psi+\eta_{2}u_{2}\psi,$$
for all $1\leq p \leq \frac{2n}{n-1}$.\\
Similarly, we can consider the invertible operator $D:L^{\frac{2n}{n-1}}(M) \to W^{-1,\frac{2n}{n-1}}(M)$, and define
$$\tilde{P}:L^{\frac{2n}{n-1}}(M) \to W^{-1,\frac{2n}{n-1}}(M)$$
by
$$\tilde{P}\tilde{v}=u_{1}\tilde{v} \; .$$
We see that in this case, we have
$$\|\tilde{P}v\|_{L^{\frac{2n}{n+1}}}\leq \|u_{1}\|_{L^{n}}\|v\|_{L^{\frac{2n}{n-1}}}.$$
Therefore, since $L^{\frac{2n}{n+1}}(M)\hookrightarrow  W^{-1,\frac{2n}{n-1}}(M)$, we have 
$$\|\tilde{P}\|_{Op}\leq C\|\psi\|_{L^{\frac{2n}{n-1}}(B_{3r})}^{2}.$$
For the same reason as above, there exists a unique solution $\tilde{v}_{0}\in L^{\frac{2n}{n-1}}(M)$, of
$$Dv=u_{1}v+\nabla \eta_{2}\cdot \psi+\eta_{2}u_{2}\psi.$$
Therefore, since $W^{1,p}(M)\hookrightarrow L^{\frac{2n}{n-1}}(M)$, for $\frac{2n}{n+1}\leq p < n$, we have that
$$v_{0}=\tilde{v}_{0}=\eta_{2}\psi \in W^{1,p}(M), \frac{2n}{n+1} \leq p < n.$$
Thus, $\psi \in W^{1,p}(M)$ for $\frac{2n}{n+1}\leq p <n$, in particular, $\psi \in L^{p}(M)$ for all $p\geq 1$. Therefore, by the elliptic regularity for $D$ and a standard bootstrap argument, we have that $\psi \in C^{\infty}(M)$.
\end{proof}

\section{Bubbling and energy quantization}

\noindent
In this section, we will analyze the behaviour of Palais-Smale sequences for $J_{g}$. This type of asymptotic study is quite standard when dealing with concentration phenomena (see for instance the books \cite{stru08, drhero04}); in particular, for our equation, the estimates that we will need are analogous to those in \cite{MaalaouiMartino2019} (Section 4; for the Dirac-Einstein problem, in dimension three) and in \cite{I} (Section 5, for the pure Dirac operator, in any dimension). We start with the first result. 

\begin{lemma}\label{lem:PSbounded}
Let $(\psi_{k})\subseteq H^{\frac{1}{2}}(M)$ be a (PS) sequence for $J_{g}$. Then $(\psi_{k})$ is bounded.
\end{lemma}

\begin{proof}
Let $(\psi_{k})$ be a (PS) sequence for $J_{g}$, at level $c\in \R$. Then
$$J_{g}(\psi_{k})=c+o(1)$$
and
$$D_{g}\psi_{k}=(G_{g}^{s}*|\psi_{k}|^{2}) \psi +\varepsilon_{k},$$
with $\varepsilon_{k}\to 0$ in $H^{-\frac{1}{2}}(M)$. Now we notice that
\begin{equation}\label{eq:J gradJ}
2J_{g}(\psi_{k})-\langle \nabla J(\psi_{k}),\psi_{k}\rangle =\frac{1}{2}\int_{M}|(G_{g}^{s}*|\psi_{k}|^{2}) \psi_{k}|^{2} \ dv.
\end{equation}
Thus,
$$\int_{M} (G_{g}^{s}*|\psi_{k}|^{2}) |\psi_{k}|^{2} \ dv= 4c+o(\|\psi_{k}\|).$$
From the elliptic regularity and the Sobolev embeddings, there exists $C>0$ such that
$$\|G_{g}^{s}*|\psi_{k}|^{2}\|_{L^{n}}^{2}\leq C\int_{M} (G_{g}^{s}*|\psi_{k}|^{2}) |\psi_{k}|^{2} \ dv.$$
On the other hand,
\begin{align}
\|\psi^{+}_{k}\|^{2}&=\int_{M}\langle \psi,\psi^{+} \rangle G_{g}^{s}*|\psi_{k}|^{2}) \ dv\notag\\
&\leq \Big(\int_{M} (G_{g}^{s}*|\psi_{k}|^{2}) |\psi_{k}|^{2} \ dv\Big)^{\frac{1}{2}} \Big(\int_{M} (G_{g}^{s}*|\psi_{k}|^{2}) |\psi_{k}^{+}|^{2} \ dv\Big)^{\frac{1}{2}}\notag\\
&\leq \Big(\int_{M} (G_{g}^{s}*|\psi_{k}|^{2}) |\psi_{k}|^{2} \ dv\Big)^{\frac{1}{2}} \|G_{g}^{s}*|\psi_{k}|^{2}\|_{L^{n}}^{\frac{1}{2}}\|\psi_{k}^{+}\|_{L^{\frac{2n}{n-1}}}\notag\\
&\leq (C+o(\|\psi_{k}\|))\|\psi_{k}^{+}\|.
\end{align}
A similar inequality holds for $\|\psi^{-}_{k}\|^{2}$, leading to
$$\|\psi_{k}\|\leq C+o(\|\psi_{k}\|).$$
Hence, $(\psi_{k})$ is bounded in $H^{\frac{1}{2}}(M)$.
\end{proof}

\noindent
\begin{remark}\label{rem:weaksolution}
From the previous Lemma, it follows that there exists $\psi_{\infty}\in H^{\frac{1}{2}}(M)$ such that (up to sub-sequences) $\psi_{k}\rightharpoonup \psi_{\infty}$ weakly in $H^{\frac{1}{2}}(M)$ and $L^{\frac{2n}{n-1}}(M)$ and strongly in $L^{p}(M)$ for $1\leq p<\frac{2n}{n-1}$. Moreover, one can easily see that $\psi_{\infty}$ is a weak solution of $(\ref{eq})$; in particular from Theorem \ref{reg} it is smooth. 
\end{remark}

\begin{lemma}\label{lem:transpose}
Let $h_{k}:=\psi_{k} - \psi_{\infty}$, then we have
$$J_{g}(h_{k}) = J_{g}(\psi_{k}) - J_{g}(\psi_{\infty})+o(1) \; ,$$
and 
$$\nabla J_{g}(h_{k})\to 0 \; .$$
\end{lemma}

\begin{proof}
We have
\begin{align}
J_{g}(\psi_{k})&=J_{g}(\psi_{\infty})+J_{g}(h_{k})+\langle \nabla J_{g}(\psi_{\infty}),h_{k}\rangle -\frac{1}{2}\int_{M} (G_{g}^{s}*|\psi_{\infty}|^{2}) |h_{k}|^{2} \ dv \notag\\
&\quad-\int_{M} (G_{g}^{s}*|h_{k}|^{2}) \langle \psi_{\infty},h_{k}\rangle \ dv -\int_{M} (G_{g}^{s}*\langle \psi_{\infty},h_{k}\rangle) \langle \psi_{\infty},h_{k}\rangle\ dv.\notag
\end{align}
Since $\psi_{\infty}$ is a solution of $(\ref{eq})$, we have $\langle \nabla J_{g}(\psi_{\infty}),h_{k}\rangle=0$. Also, since $h_{k}\to 0$ weakly in $H^{\frac{1}{2}}(M)$ we have that $h_{k}\to 0$ strongly in $L^{p}(M)$ for all $p<\frac{2n}{n-1}$. Therefore, 
$$\int_{M} (G_{g}^{s}*|\psi_{\infty}|^{2})|h_{k}|^{2}\ dv\to 0, \text{ as } k\to \infty.$$
Similarly, we have by H\"{o}lder's inequalities that
\begin{align}
\left|\int_{M} (G_{g}^{s}*|h_{k}|^{2}) \langle \psi_{\infty},h_{k}\rangle \ dv \right| &\leq \|G_{g}^{s}*|h_{k}|^{2}\|_{L^{n}}\|\psi_{\infty}\|_{L^{\infty}}\|h_{k}\|_{L^{\frac{n}{n-1}}}\notag\\
&\leq C \|h_{k}\|_{L^{\frac{2n}{n-1}}}^{2}\|h_{k}\|_{L^{\frac{n}{n-1}}}.\notag
\end{align}
Notice that Lemma \ref{lem:PSbounded} implies that $\|h_{k}\|_{H^{\frac{1}{2}}}$ is uniformly bounded and $\|h_{k}\|_{L^{\frac{n}{n-1}}}\to 0$ as $k\to \infty$. Hence, we have
$$\int_{M} (G_{g}^{s}*|h_{k}|^{2}) \langle \psi_{\infty},h_{k}\rangle \ dv \to 0 \text{ as } k\to \infty.$$
Similarly,
$$\int_{M} (G_{g}^{s}*\langle \psi_{\infty},h_{k}\rangle) \langle \psi_{\infty},h_{k}\rangle\ dv\to 0, \text{ as } k\to \infty.$$
Thus,
$$J_{g}(\psi_{k})=J_{g}(\psi_{\infty})+J_{g}(h_{k})+o(1).$$
The statement for $\nabla J_{g}(h_{k})$ can be proved in the same way.  
\end{proof}

\noindent
From now on, we will assume without loss of generality that the (PS) sequence $(\psi_{k})$ converges weekly to $0$, namely $\psi_{\infty}=0$. Given $\varepsilon_{0}>0$, we define the following sets 
$$\Sigma_{1}(\varepsilon_{0})=\left\{x\in M; \liminf_{r\to 0}\liminf_{k\to \infty}\int_{B_r(x)}|\psi_{k}|^{\frac{2n}{n-1}}\ dv>\varepsilon_{0} \right\},$$
$$\Sigma_{2}(\varepsilon_{0})=\left\{x\in M; \liminf_{r\to 0}\liminf_{k\to \infty}\int_{B_r(x)}\Big(G_{g}^{s}*|\psi_{k}|^{2}\Big)^{n}\ dv>\varepsilon_{0} \right\},$$
and 
$$\Sigma_{3}(\varepsilon_{0})=\left\{x\in M; \liminf_{r\to 0}\liminf_{k\to \infty}\int_{B_r(x)} (G_{g}^{s}*|\psi_{k}|^{2}) |\psi_{k}|^{2} \ dv>\varepsilon_{0} \right\},$$
where $B_r(x)$ is the geodesic ball with center in $x$ and radius $r$. We can state then the following $\varepsilon$-regularity type result.
\begin{lemma}
Let $(\psi_{k})$ be a (PS) sequence converging weekly to $0$. There exists $\varepsilon_{0}>0$ such that if $x\not \in \Sigma_{1}(\varepsilon_{0})\cap \Sigma_{2}(\varepsilon_{0})\cap \Sigma_{3}(\varepsilon_{0})$, then there exists $r>0$ such that $\psi_{k}\to 0$ in $H^{\frac{1}{2}}(B_{r}(x_{0}))$.
\end{lemma}

\begin{proof}
We will use the same notations as in the proof of Theorem \ref{reg}. We have, as in (\ref{eqreg}),
$$D_{g}(\eta_{2} \psi_{k})=\eta_2 (G_{g}^{s}*|\psi_{k}|^{2})\psi_{k}+\nabla \eta_{2} \cdot \psi_{k}+\delta_{k},$$
where  $\delta_{k}\to 0$ in $H^{-\frac{1}{2}}(M)$. Using elliptic estimates, we have 
\begin{align}
\|\eta_{2} \psi_{k}\|_{H^{\frac{1}{2}}}&\leq C\|\eta_{2} (G_{g}^{s}*|\psi_{k}|^{2}) \psi_{k}+\nabla \eta_{2} \cdot \psi_{k}+\delta_{k}\|_{H^{-\frac{1}{2}}}\notag\\
&\leq C_{1} \Big( \|\eta_{2} (G_{g}^{s}*|\psi_{k}|^{2}) \psi_{k}\|_{L^{\frac{2n}{n+1}}(B_{r})}+\|\nabla \eta_{2} \cdot \psi_{k}\|_{L^{\frac{2n}{n+1}}}+o(1)\Big).\notag
\end{align}
In addition, we have
$$\|\nabla \eta_{2} \cdot \psi_{k}\|_{L^{\frac{2n}{n+1}}}\leq C_2 \|\psi_{k}\|_{L^{\frac{2n}{n+1}}}\to 0 \; .$$
Now, we assume first that $x_0\not \in \Sigma_{2}(\varepsilon)$, then by H\"{o}lder's inequalities,
\begin{align}
\|\eta_{2} (G_{g}^{s}*|\psi_{k}|^{2}) \psi_{k}\|_{L^{\frac{2n}{n+1}}}&\leq  \|G_{g}^{s}*|\psi_{k}|^{2}\|_{L^{n}(B_{r})}\|\eta_{2} \psi_{k}\|_{L^{\frac{2n}{n-1}}}\notag\\
&\leq C_{3}\varepsilon^{\frac{1}{n}}\|\eta_{2} \psi_{k}\|_{H^{\frac{1}{2}}}\notag \;.
\end{align}
Therefore, if $ C_{1}C_{3}\varepsilon^{\frac{1}{n}}<\frac{1}{2}$, we have
$$\|\eta_{2} \psi_{k}\|_{H^{\frac{1}{2}}}\leq C_4 \|\nabla \eta_{2} \cdot \psi_{k}\|_{L^{\frac{2n}{n+1}}}+o(1).$$
Hence, $\eta_{2} \psi_{k}\to 0$ in $H^{\frac{1}{2}}(M)$.\\
On the other hand, let us assume that $x_0\not \in \Sigma_{1}(\varepsilon)$. Then as in (\ref{eqreg2}), we can write 
\begin{equation}\label{eq:epsreg1}
  \eta_{2} (G_{g}^{s}*|\psi_{k}|^{2}) \psi_{k}= \eta_{2} (G_{g}^{s}*|\eta_{1}\psi_{k}|^{2}) \psi_{k}+ \eta_{2}(G_{g}^{s}*(1-\eta_{1}^{2})|\psi_{k}|^{2}) \psi_k =A_{1}(\psi_{k})+A_{2}(\psi_{k}).
\end{equation}
Now we notice that
$$\|A_{2}(\psi_{k})\|_{L^{\frac{2n}{n+1}}}\leq C_5\|\psi_{k}\|_{L^{2}}^{3} \to 0 \text{ as } k\to 0.$$
Also,
\begin{align}
\|A_{1}(\psi_{k})\|_{L^{\frac{2n}{n+1}}}&\leq C_{6}\|\eta_{1}\psi_{k}\|^{2}_{L^{\frac{2n}{n-1}}}\|\eta_{2}\psi_{k}\|_{L^{\frac{2n}{n-1}}}\notag\\
&\leq C_{6}C_{7}C_{8}\varepsilon^{\frac{n-1}{n}}\|\eta_{2}\psi_{k}\|_{H^{\frac{1}{2}}}.\notag
\end{align}
Hence, for $C_{6}C_{7}C_{8}\varepsilon^{\frac{n-1}{n}}<\frac{1}{2}$, we get again
$$\|\eta_{2}\psi_{k}\|_{H^{\frac{1}{2}}}\to 0.$$
In order to finish the proof, we see from the decomposition (\ref{eq:epsreg1}) that for $x_{0}\not \in \Sigma_{3}(\varepsilon)$ we have
\begin{align}
\|A_{1}(\psi_{k})\|_{L^{\frac{2n}{n+1}}}&\leq \Big(\int_{M}\eta_{2} (G_{g}^{s}*|\eta_{1}\psi_{k}|^{2}) |\psi_{k}|^{2} \ dv\Big)^{\frac{1}{2}}\|G_{g}^{s}*|\eta_{1}\psi_{k}|^{2}\|_{L^{n}}^{\frac{1}{2}}\notag\\
&\leq C_9 \varepsilon^{\frac{1}{2}}\|\eta_{2}\psi_{k}\|_{H^{\frac{1}{2}}}\notag
\end{align}
Again, for $ C_9\varepsilon^{\frac{1}{2}}<\frac{1}{2}$ we obtain the desired conclusion.  
\end{proof}

\noindent
As a corollary, we get the following result.
\begin{proposition}\label{pro:PSstrong}
Let $(\psi_{k})$ be a (PS) sequence converging weekly to $0$. If the (PS) sequence $(\psi_{k})$ does not converge (up to subsequences) strongly to zero in $H^{\frac{1}{2}}(M)$, then there exists $\varepsilon_{0}>0$ (even smaller if necessary) such that
$$\Sigma_{1}(\varepsilon_{0})=\Sigma_{2}(\varepsilon_{0})= \Sigma_{3}(\varepsilon_{0})\not=\emptyset.$$
Moreover, if $(\psi_{k})$ is a (PS) sequence at level $c$, with $4c < \varepsilon_{0}$, then  $(\psi_{k})$ converges strongly to zero in $H^{\frac{1}{2}}(M)$.
\end{proposition}

\noindent
The last assertion follows immediately from equation (\ref{eq:J gradJ}) and the definition of $\Sigma_{3}(\varepsilon_{0})$.

\noindent
We need here to take into account again the dependence on the metric. Let us consider now the concentration function 
$$Q_{k}(t)=\sup_{x\in M} \int_{B_t(x)} (G_g^{s}*|\psi_{k}|^{2})^{n} \ dv_{g} .$$
If we assume that $\Sigma_{2}(\varepsilon_{0})\not=0$, then given $\varepsilon>0$ so that $3\varepsilon<\varepsilon_{0}$, there exists $R_{k}>0$ such that $R_{k}\to 0$ and a sequence $x_{k}\in M$, that we can assume converging to a certain $x_{0}\in \Sigma_{2}(\varepsilon_{0})$ so that
\begin{equation}\label{eq:Q concentration epsilon}
  Q_{k}(R_{k})=\int_{B_{R_{k}}(x_{k})} (G_g^s*|\psi_{k}|^{2})^{n} dv_{g} =\varepsilon.
\end{equation}
We let $\rho_{k}(x)=\exp_{x_{k}}(R_{k}x)$ defined for $R_{k}|x|<\iota(M)$; here $\iota(M)$ is the injectivity radius of $M$, that we will assume for the sake of simplicity $\iota(M)\geq 3$. Therefore, if we let $B_{R}^{0}$ denote the Euclidean ball centered at zero and of radius $R$, then we have that the two spaces $(B_{R}^{0}, R_{k}^{-2}\rho^{*}_{k}g)$ and $(B_{R_{k}R}(x_{k}),g)$ are conformally equivalent for $k$ large enough. We define then the metric $g_{k}=R_{k}^{-2}\rho_{k}^{*}g$ on $B_{R}^{0}$. It is easy to see that $g_{k}\to g_{\R^{n}}$ in $C^{\infty}(B_{R}^{0})$. We will use the map $\rho_{k}$ to also identify the spinor bundles, that is 
$$(\rho_{k})_{*}:\Sigma_{x_{0}}(B_{R}^{0},g_{k})\to \Sigma_{\rho_{k}(x_{0})}(M,g) .$$
We can then define the spinor $\Psi_{k}=R_{k}^{\frac{n-1}{2}}\rho_{k}^{*}\psi_{k}$ on $\Sigma_{x_{0}}(B_{R}^{0},g_{k})$,  where $\rho_{k}^{*}\psi_{k}=(\rho_{k})_{*}^{-1}\circ \psi_{k}\circ (\rho_{k})_{*}$. Therefore, based on the properties of the convolution and the conformal invariance, we have 
$$\int_{B_{R}^{0}\times B_{R}^{0}}G_{g_{k}}^s(x,y) |\Psi_{k}(x)|^{2} |\Psi_{k}(y)|^{2} \ dv_{g_{k}}(x) \ dv_{g_{k}}(y)$$
$$=\int_{B_{R_{k}R}(x_{k})\times B_{R_{k}R}(x_{k})}G_{g}^s(x,y) |\psi_{k}(x)|^{2} |\psi_{k}(y)|^{2} \ dv_{g}(x) \ dv_{g}(y),$$

$$\int_{B_{R}^{0}}\langle \Psi_{k},D_{g_{k}}\Psi_{k}\rangle\ dv_{g_{k}}=\int_{B_{R_{n}R}(x_{k})}\langle \psi_{k},D_{g}\psi_{k}\rangle\ dv_{g},$$
and 
\begin{equation}\label{eq:Psi-psi}
  \int_{B_{R}^{0}}|\Psi_{k}|^{\frac{2n}{n-1}}\ dv_{g_{k}}=\int_{B_{R_{n}R}(x_{k})}|\psi_{k}|^{\frac{2n}{n-1}}\ dv_{g}.
\end{equation}

\noindent
We can now, state the following result.
\begin{proposition}\label{pro:Fk}
Let $\Psi_{k}$ be the spinor on $\Sigma_{x_{0}}(B_{R}^{0},g_{k})$, defined as before. Let us set
$$F_{k}:=D_{g_{k}}\Psi_k-(G_{g_{k}}^s*|\Psi_{k}|^{2})\Psi_{k} .$$
Then $F_{k}\to 0$ in $H^{-\frac{1}{2}}_{loc}(\R^{n})$, namely for $R>0$, it holds
\begin{equation}\label{eq:H^-1/2 loc}
  \sup\Big\{\int_{\R^{n}}\langle F_{k},\Phi\rangle\ dv_{g_{k}} \; ;  \Phi \in H^{\frac{1}{2}}(\R^{n}) \; , \|\Phi\|_{H^{\frac{1}{2}}(\R^{n})}\leq 1 \; , supp(\Phi)\subset B_{R}^{0}\Big\} \to 0 \text{ as } k\to \infty.
\end{equation}
\end{proposition}

\begin{proof}
From the definition of $F_k$ and the conformal invariance, we have
$$F_{k}=D_{g_{k}}\Psi_k-(G_{g_{k}}^s*|\Psi_{k}|^{2})\Psi_{k}= R_{k}^{\frac{n+1}{2}}\rho^{*}_{k}\left(D_{g}\psi_{k}-(G_{g}^s*|\psi_{k}|^{2})\psi_{k}\right)= R_{k}^{\frac{n+1}{2}}\rho_{k}^{*}\delta_{k}.$$
Consider then a test spinor $\Phi$ with $supp(\Phi)\subset B_{R}^{0}$ and $\|\Phi\|_{H^{\frac{1}{2}}}\leq 1$. Then we have
\begin{align}
\int_{B_{R_{k}^{-1}}^{0}}\langle F_{k},\Phi\rangle\ dv_{g_{k}}&=\int_{B_{R_{k}^{-1}}^{0}}\langle \rho_{k}^{*}\delta_{k},R_{k}^{-\frac{n-1}{2}}\Phi\rangle\ dv_{\rho_{k}^{*}g}\notag\\
&=\int_{B_{1}(x_{k})}\langle \delta_{k}, R_{k}^{-\frac{n-1}{2}}(\rho_{k}^{-1})^{*}\Phi\rangle\ dv_{g}.\notag
\end{align}
Since $\|\Phi\|_{H^{\frac{1}{2}}(\R^{n})}\leq 1$, there exists $C>0$ such that $\|R_{k}^{-\frac{n-1}{2}}(\rho_{k}^{-1})^{*}\Phi\|_{H^{\frac{1}{2}}(M)}\leq C$. Hence we have that (\ref{eq:H^-1/2 loc}) holds.
\end{proof}

\noindent
We introduce here the following space
$$D^{\frac{1}{2}}(\R^{n})=\left\{\Phi \in L^{\frac{2n}{n-1}}(\R^{n})\; ; |\xi|^{\frac{1}{2}}|\widehat{\Phi}|\in L^{2}(\R^{n})\right\},$$
where $\widehat{\Phi}$ is the Fourier transform of $\Phi$.

\begin{proposition}
Let $\varepsilon>0$ small enough in (\ref{eq:Q concentration epsilon}), then there exists $\Psi_\infty\in D^{\frac{1}{2}}(\R^{n})$ such that $\Psi_{k}\to \Psi_\infty$ in $H^{\frac{1}{2}}_{loc}(\R^{n})$ and $\Psi_\infty$ satisfies the equation
\begin{equation}\label{eq:equationRn}
D_{g_{\R^{n}}}\Psi_\infty=(G_{g_{\R^{n}}}^s*|\Psi_\infty|^{2})\Psi_\infty, \quad \text{in } \R^n \; .
\end{equation}
\end{proposition}

\begin{proof}
First, the sequence $\Psi_k$ is bounded in $H^{\frac{1}{2}}_{loc}(\R^{n})$, hence there exists $\Psi_{\infty}$ such that, up to subsequence, $\Psi_{k}\rightharpoonup \Psi_{\infty}$ in $H_{loc}^{\frac{1}{2}}(\R^{n})$ and strongly in $L_{loc}^{p}(\R^{n})$ for $1\leq p<\frac{2n}{n-1}$.\\
Now, from the relation (\ref{eq:Psi-psi}), we have 
\begin{equation*}
\limsup_{k\to \infty} \int_{B_{R}^{0}}|\Psi_{k}|^{\frac{2n}{n-1}}\ dv_{g_{k}} \leq \sup_{k\geq 1}  \int_{M}|\psi_{k}|^{\frac{2n}{n-1}}\ dv_{g} < +\infty ,
\end{equation*}
hence $\Psi_{\infty}\in L^{\frac{2n}{n-1}}(\R^{n})$.\\
Next, arguing as in Lemma \ref{lem:PSbounded} and Remark \ref{rem:weaksolution}, we have that $\Psi_{\infty}$ is a weak solution of (\ref{eq:equationRn}), from which we deduce that $\Psi_{\infty}\in D^{\frac{1}{2}}(\R^{n})$.\\
We can now assume without loss of generality that $\Psi_{\infty}=0$, just replacing $\Psi_{n}$ by $\Psi_{n}-\Psi_{\infty}$ and using Lemma \ref{lem:transpose}.\\
But by assumption we have that, given $x\in \R^{3}$, for $k$ big enough we get from (\ref{eq:Q concentration epsilon}) 
$$\int_{B_{1}^{0}}  (G_{g_k}^s*|\Psi_{k}|^{2})^{n} dv_{g_k}  =\int_{B_{R_{k}}(x_{k})} (G_g^s*|\psi_{k}|^{2})^{n} dv_{g}  =\varepsilon . $$
Let $\beta \in C^{\infty}_{0}(\R^{n})$, with $supp(\beta)\in B_{1}^{0}$, we get 
\begin{align}
\|\beta \Psi_{k}\|_{H^{\frac{1}{2}}}&\leq C\left(\|D_{g_{\R^{n}}}(\beta \Psi_{k})\|_{H^{-\frac{1}{2}}}+\|\beta \Psi_{k}\|_{L^{2}}\right)\\
&\leq C\left(\|D_{g_{k}}(\beta \Psi_{k})\|_{H^{-\frac{1}{2}}} + \|(D_{g_{\R^{n}}}-D_{g_{k}})(\beta \Psi_{k})\|_{H^{-\frac{1}{2}}} + \|\beta \Psi_{k}\|_{L^{2}}\right).\notag
\end{align}
We have $\|\beta \Psi_{k}\|_{L^{2}}\to 0$ and, since $g_{k}\to g_{\R^{n}}$ in $C^{\infty}$, we also have $\|(D_{g_{\R^{n}}}-D_{g_{k}})(\beta \Psi_{k})\|_{H^{-\frac{1}{2}}}\to 0.$
For the last term, using Proposition \ref{pro:Fk}, we have
$$\|D_{g_{k}}(\beta \Psi_{k})\|_{H^{-\frac{1}{2}}}  \leq \|(G_{g_{k}}^s*|\Psi_{k}|^{2})\beta\Psi_{k} + \beta F_{k})\|_{H^{-\frac{1}{2}}}+o(1),$$
hence
$$\|D_{g_{k}}(\beta \Psi_{k})\|_{H^{-\frac{1}{2}}}  \leq \|(G_{g_{k}}^s*|\Psi_{k}|^{2})\beta\Psi_{k} )\|_{H^{-\frac{1}{2}}}+o(1),$$
since $\beta F_{k}\to 0$ in $H^{-\frac{1}{2}}$. Finally
\begin{align}
\|\beta \Psi_{k}\|_{H^{\frac{1}{2}}}  & \leq C \|(G_{g_{k}}^s*|\Psi_{k}|^{2})\beta\Psi_{k} )\|_{L^{\frac{2n}{n+1}}}+o(1) \notag \\
&\leq C \|G_{g_k}^{s}*|\Psi_{k}|^{2}\|_{L^{n}}\|\beta \Psi_{k}\|_{L^{\frac{2n}{n-1}}} +o(1)\notag\\
&\leq C\varepsilon^{\frac{1}{n}}\|\beta \Psi_{k}\|_{L^{\frac{2n}{n-1}}} +o(1) \to 0 \notag \; .
\end{align}
\end{proof}

\noindent
We observe that by the regularity Theorem \ref{reg}, we have that indeed $\Psi_\infty\in C^{\infty}(\R^{n})$.  Now, for $\varepsilon>0$ and small enough, as before, there exists $R_{k}>0$ such that $R_{k}\to 0$ and a sequence $x_{k}\in M$, that we can assume converging to $x_{0}\in M$. We consider then a cut-off function $\beta=1$ on $B_{1}(x_{0})$ with $supp(\beta)\subset B_{2}(x_{0})$ and we define $\phi_{k}\in C^{\infty}(M)$ by
\begin{equation}\label{e6}
\phi_{k}=R_{k}^{-\frac{n-1}{2}}\beta (\rho_{k}^{-1})^{*}(\Psi_\infty) .
\end{equation}
We have then the last result of this section:
\begin{lemma}\label{lem:estimates}
Let $(\psi_{k})$ be a (PS) sequence and set $\overline{\psi}_{k}=\psi_{k}-\phi_{k}$. Then, up to a subsequences, 
\begin{equation}\label{eq:psibar}
\overline{\psi}_{k}\rightharpoonup 0 \text{ in } H^{\frac{1}{2}}( M) ;
\end{equation}
\begin{equation}\label{eq:gradpsibar}
\nabla J_{g}(\phi_{k}) \to 0 \text{ and } \; \nabla J_{g}(\overline{\psi}_{k})  \to 0, \text{ in } H^{-\frac{1}{2}}(M).
\end{equation}
Moreover, we have the following energy estimate
\begin{equation}\label{eq:energy}
J_{g}(\overline{\psi}_{k})=J_{g}(\psi_{k})-J_{g_{\R^{n}}}(\Psi_{\infty})+o(1).
\end{equation}
\end{lemma}

\begin{proof}
The proofs of these last three estimates are similar to those in \cite{MaalaouiMartino2019} (Lemma 4.8, 4.9, 4.10) and in \cite{I} (Lemma 5.6, 5.7, 5.8); for this reason, in order to show how to handle our nonlinearity in this situation, we will prove only the first limit (\ref{eq:psibar}).\\
We know that $\psi_{k}\rightharpoonup 0$, therefore we need to study the weak convergence of $\phi_{k}$; now, we know that this is bounded in $H^{\frac{1}{2}}( M)$, thus up to subsequences, it has a weak limit: we have to prove that the limit is zero. In particular, given a test spinor $h\in C^{\infty}( M)$, we will show that
\begin{equation}\label{eq:phikweakzero}
  \int_{M}\langle \phi_{k},h\rangle dv_{g}\to 0.
\end{equation}
Let us fix $R>0$, we will prove two estimates, the first one on $B_{R_{k}R}(x_{k})$ and the second one on $M\setminus B_{R_{k}R}(x_{k})$.
By definition of $\phi_k$ given in (\ref{e6}) and the conformal change, we have
\begin{align}
\int_{B_{R_{k}R}(x_{k})}\langle \phi_{k},h\rangle dv_{g} & = R_{k}^{-\frac{n-1}{2}} \int_{B_{R_{k}R}(x_{k})} \beta \langle (\rho_{k}^{-1})^{*}(\Psi_\infty) ,h\rangle dv_{g} \notag\\
&=  R_{k}^{\frac{n+1}{2}} \int_{B_{R}^{0}} \rho_{k}^{*}(\beta)   \langle \Psi_\infty ,  \rho_{k}^{*}(h)\rangle  dv_{g_{k}} . \notag
\end{align}
Therefore,
$$\left| \int_{B_{R_{k}R}(x_{k})}\langle \phi_{k},h\rangle dv_{g} \right|\leq C_1 R_{k}^{\frac{n+1}{2}} \|h\|_{\infty} \int_{B_{R}^{0}}|\Psi_\infty|dv_{g_{\R^n}}.$$
In the same way, if $k$ is large enough,  we have 
\begin{align}
\int_{M \setminus B_{R_{k}R}(x_{k})}\langle \phi_{k},h\rangle dv_{g} & =  R_{k}^{\frac{n+1}{2}} \int_{B_{3R_{k}^{-1}}^{0}\setminus B_{R}^{0}} \rho_{k}^{*}(\beta)   \langle \Psi_\infty ,  \rho_{k}^{*}(h)\rangle  dv_{g_{k}}   .\notag
\end{align}
Thus,
\begin{align}
\left| \int_{M \setminus B_{R_{k}R}(x_{k})}\langle \phi_{k},h\rangle dv_{g} \right| & \leq C_2 R_{k}^{\frac{n+1}{2}} \|h\|_{\infty} \int_{B_{3R_{k}^{-1}}^{0}\setminus B_{R}^{0}} |\Psi_\infty|dv_{g_{\R^n}} \notag\\
& \leq C_3 \|h\|_{\infty} \left(\int_{B_{3R_{k}^{-1}}^{0}\setminus B_{R}^{0}}  |\Psi_{\infty}|^{\frac{2n}{n-1}}  dv_{g_{\R^{n}}}\right)^{\frac{n-1}{2n}}.\notag
\end{align}
Finally, we put the two estimates together and we get
$$\left| \int_{M}\langle \phi_{k},h\rangle dv_{g} \right|\leq C \|h\|_{\infty}\left( R_{k}^{\frac{n+1}{2}} \int_{B_{R}^{0}}|\Psi_\infty|dv_{g_{\R^n}}  + \left(\int_{B_{3R_{k}^{-1}}^{0}\setminus B_{R}^{0}}  |\Psi_{\infty}|^{\frac{2n}{n-1}}  dv_{g_{\R^{n}}}\right)^{\frac{n-1}{2n}}   \right).$$
Therefore, if $k\to \infty$ and then $R\to \infty$, we obtain (\ref{eq:phikweakzero}). 
\end{proof}

\noindent
\begin{remark}\label{rem:Sphere-R^n}
In order to finalize the proof of Theorem \ref{thmbubble}, we need a last estimate regarding the solutions of equation (\ref{eq:equationRn}). Since Theorem \ref{leasten} addresses an explicit conformal lower bound of the energy of ground state solutions, we will use the result and leave the details to the next section. Here we just notice the following fact: if $\Psi_\infty$ satisfies equation (\ref{eq:equationRn}), then its pull-back by the standard stereographic projection satisfies equation (\ref{sol sphere}). Therefore, by (\ref{eq:low bound sol sphere}) we have that there exists a positive constant $C_{\R^{n}}$, such that
\begin{equation}\label{eq:lowbound R^n}
J_{g_{\R^{n}}}(\Psi_{\infty})\geq C_{\R^{n}} .
\end{equation}
\end{remark}

\noindent
\begin{proof}  (\emph{of Theorem (\ref{thmbubble})})\\
Let $(\psi_{k})$ be a Palais-Smale sequence for $J_{g}$ at level $c$; we will apply a standard iteration procedure. Let 
$$\psi^1_{k}:=\psi_{k} -\psi_{\infty}, $$ 
then by Lemma \ref{lem:transpose} we have
$$J_{g}(\psi^1_{k}) = J_{g}(\psi_{k}) - J_{g}(\psi_{\infty})+o(1) \; .$$
As we saw after Proposition \ref{pro:PSstrong}, we can find a sequence of points $x_{k}^{1} \in M$ converging to some point $x^{1}\in M$, a sequence of real numbers $R_{k}^{1}$ converging to zero, a function $\Psi^1_{\infty}$ solution of (\ref{eq:equationRn}) and its related $\phi^1_{k}$ defined as in (\ref{e6}). Next we define
$$\psi^2_{k}:=\psi^1_{k} - \phi^1_{k} = \psi_{k} - \psi_{\infty} - \phi^1_{k} \; . $$
By equation (\ref{eq:energy}) in Lemma \ref{lem:estimates}, we obtain
$$ J_{g}(\psi^2_{k}) = J_{g}(\psi^1_{k}) - J_{g_{\R^{n}}}(\Psi^1_{\infty})+o(1) = J_{g}(\psi_{k}) - J_{g}(\psi_{\infty})- J_{g_{\R^{n}}}(\Psi^1_{\infty})+o(1) \; . $$
We can repeat this procedure $m$ times, finding $m$ sequences of points $x_{k}^{1},\cdots, x_{k}^{m} \in M$ converging to some points $x^{1},\cdots, x^{m}\in M$, $m$ sequences of real numbers $R_{k}^{1},\cdots, R_{k}^{m}$ converging to zero, $m$ functions $\Psi^1_{\infty},\cdots, \Psi^m_{\infty}$ solutions of (\ref{eq:equationRn}) and the related $\phi^1_{k} ,\cdots, \phi^m_{k}$ defined as in (\ref{e6}), with
$$\psi^{m+1}_{k}:= \psi_{k} - \psi_{\infty} - \sum_{j=1}^{m}\phi_{k}^{j} \; , $$
$$ J_{g}(\psi^{m+1}_{k}) = J_{g}(\psi_{k}) - J_{g}(\psi_{\infty})- \sum_{j=1}^{m}J_{g_{\mathbb{R}^{n}}}(\Psi_\infty^{j}) +o(1) \; . $$
Now, from (\ref{eq:lowbound R^n}) in Remark \ref{rem:Sphere-R^n}, we have that 
$$J_{g_{\mathbb{R}^{n}}}(\Psi_\infty^{j}) \geq C_{\R^{n}} \quad j=1,\ldots , m \; .$$
Therefore, since from Proposition \ref{pro:PSstrong} (PS) sequences at levels strictly below $ \frac{\varepsilon_{0}}{4}$ converge strongly to zero in $H^{\frac{1}{2}}(M)$, we stop the iteration when $c-m C_{\R^{n}} < \frac{\varepsilon_{0}}{4}$, obtaining the thesis.
\end{proof}

\section{Least Energy Solution on the Sphere}

In this section, we will provide the proof of Theorem \ref{leasten}.\\ 
We start by recalling the conformal invariant $\lambda^{+}(M, [g])$, which was thoroughly studied in\cite{Am1,Am2} in order to study the optimal first eigenvalues of the Dirac operator. It plays the same role as the classical Yamabe invariant but for its spinorial version (we also refer the reader to \cite{YX} for recent results on the spinorial Yamabe problem). One way of defining $\lambda^{+}(M,[g])$ is as follows:
$$\lambda^{+}(M,[g]):=\inf\{\lambda_{1}(D_{h})Vol(h)^{\frac{1}{n}}, h\in [g]\}.$$
It can also be characterized by
$$\lambda^{+}(M, [g])=\inf_{\psi \in C^{\infty}(\Sigma M); \langle D_{g}\psi,\psi\rangle \not=0} \frac{\Big(\int_{M}|D_{g}\psi|^{\frac{2n}{n+1}}\ dv \Big)^{\frac{n+1}{n}}}{|\int_{M} \langle D_{g}\psi,\psi\rangle \ dv |}.$$

\noindent
\begin{proof}  (\emph{of Theorem (\ref{leasten})})\\
If $\psi$ is a non-trivial solution of $(\ref{sol sphere})$, then we have
\begin{align}
\lambda^{+}(S^{n}, [g_{0}])&\leq \frac{\Big(\int_{S^{n}}\Big(|\psi|G^{s}*|\psi|^{2}\Big)^{\frac{2n}{n+1}}\ dv\Big)^{\frac{n+1}{n}}}{\int_{S^{n}}G^{s}*|\psi|^{2} |\psi|^{2}\ dv}\notag\\
&\leq \Big(\int_{S^{n}}\Big(G^{s}*|\psi|^{2}\Big)^{n}\ dv\Big)^{\frac{1}{n}}.\notag
\end{align}
But $u=G^{s}*|\psi|^{2}$ satisfies $P_{g_{0}}^{s}u=|\psi|^{2}$. So if we define the $H^{s}$-norm by 
$$\|u\|_{H^{s}}=\|(P_{g}^{s})^{\frac{1}{2}}u\|_{L^{2}}=\|uP_{g}^{s}u\|_{L^{1}}^{\frac{1}{2}} ,$$
as in Section 2, we have from the Sobolev embedding $H^{s}(M)\hookrightarrow L^{\frac{2n}{n-2s}}(M)=L^{n}(M)$, that
$$\|u\|_{L^{n}}\leq Y_{s}(S^{n}, [g_{0}])^{\frac{1}{2}}\|u\|_{H^{s}},$$
where $Y_{s}(S^{n}, [g_{0}])$ is defined in $(\ref{yamc})$. Thus,
\begin{equation}\label{sob}
\Big(\int_{S^{n}}\Big(G^{s}*|\psi|^{2}\Big)^{n}\ dv\Big)^{\frac{1}{n}}\leq \frac{1}{Y_{s}(S^{n}, [g_{0}])^{\frac{1}{2}}} \Big(\int_{S^{n}}|\psi|^{2}G^{s}*|\psi|^{2}\ dv \Big)^{\frac{1}{2}}.
\end{equation}
In particular, we have that
\begin{equation}\label{low}
4J_{g_0}(\psi)\geq \lambda^{+}(S^{n}, [g_{0}])^{2} Y_{s}(S^{n}, [g_{0}]).
\end{equation}

\noindent
We assume now that $\psi$ is a ground state solution on $(S^{n}, g_{0})$. Then we are in the case of equality in H\"{o}lder's inequalities. Namely,
$$\int_{S^{n}} \Big(|\psi|^{2}G^{s}*|\psi|^{2}\Big)^{\frac{n}{n+1}}\Big(G^{s}*|\psi|^{2}\Big)^{\frac{n}{n+1}} \ dv =\Big(\int_{S^{n}}|\psi|^{2}G^{s}*|\psi|^{2}\ dv\Big)^{\frac{n}{n+1}}\Big(\int_{S^{n}}(G^{s}*|\psi|^{2})^{n}\ dv \Big)^{\frac{1}{n}}.$$
Hence, for $c_{n}=\frac{\int_{S^{n}}(G^{s}*|\psi|^{2})^{n}\ dv}{\int_{S^{n}}|\psi|^{2}G^{s}*|\psi|^{2}\ dv}$, we have
\begin{equation}\label{equality}
c_{n}|\psi|^{2}=\Big(G^{s}*|\psi|^{2}\Big)^{n-1}.
\end{equation}
From the equalities in $(\ref{sob})$ and $(\ref{low})$, we have
$$c_{n}^{\frac{1}{n}}Y_{s}(S^{n},[g_{0}])^{\frac{1}{2}}=\Big(\lambda^{+}(S^{n}, [g_{0}])^{2} Y_{s}(S^{n}, [g_{0}])\Big)^{\frac{n-2}{2n}},$$
Thus $$c_{n}=\frac{\lambda^{+}(S^{n}, [g_{0}])^{n-2}}{ Y_{s}(S^{n}, [g_{0}])}.$$
On the other hand, from $(\ref{equality})$, we have that the function $u=G^{s}*|\psi|^{2}$ satisfies 
$$P_{G_{\R^{n}}^{s}}^{s}u=\frac{1}{c_{n}}u^{n-1}$$
Hence, by the classification results in \cite{CLO}, we have that up to a conformal change, $u$ is constant and hence $|\psi|^{2}$ is constant. In particular, from the case of equality in Hijazi's inequality \cite{Hij1,Hij2}, $\psi$ is a $-\frac{1}{2}$-Killing Spinor on $S^{n}$.
\end{proof}

\noindent
From the conformal invariance of $(\ref{eq})$ and $(\ref{energyfunctional})$, we also have the following
\begin{corollary}
Let $\psi\in H^{\frac{1}{2}}(\R^{n},\mathbb{C}^{N})$ be a non-trivial ground state solution for the equation
$$D_{\R^{n}}\psi=G_{\R^{n}}^{s}*|\psi|^{2}\psi,$$
where $G_{\R^{n}}^{s}$ is the Green's function of the Laplacian on $\R^{n}$, with $N=2^{[\frac{n}{2}]}$. Then there exists $\Phi_{0}\in \mathbb{C}^{N}$, a point $x_{0}\in \R^{n}$ and $\lambda>0$ so that
$$\psi(x)=c_{n}\Big(\frac{\lambda}{\lambda^{2}+|x-x_{0}|^{2}}\Big)^{\frac{n}{2}}\Big(1-\Big(\frac{x-x_{0}}{\lambda}\Big)\cdot \Phi_{0}\Big).$$
\end{corollary}

\section{Brezis-Nirenberg Problem}

\noindent
We focus now on the linearly perturbed problem
$$D_{g}\psi =\lambda \psi +G^{s}*|\psi|^{2} \psi,$$
where $\lambda>0$ and not a spectral value of $D_{g}$. Also, define $H^{+}_{\lambda}$, $H^{-}_{\lambda}$ and $H^{0}_{\lambda}$ to be the positive, negative and null space of $D_{g}-\lambda$ on $H^{\frac{1}{2}}(\Sigma M)$. Notice that zero is a trivial solution for the problem. In fact, one can obtain solutions to the problem if $\lambda<\lambda_{k+1}\in Spec(D_{g})$ and $\lambda$ close to $\lambda_{k+1}$. Indeed, this type of solutions can be obtained using bifurcation theory. For instance, if we define the operator $L_{\lambda}:H^{\frac{1}{2}}(M)\to H^{\frac{1}{2}}(M)$ by
$$L_{\lambda}\psi=(1+|D_{g}|)^{-1}(D_{g}\psi)-(1+|D_{g}|)^{-1}(\lambda \psi+G^{s}*|\psi|^{2}\psi).$$
Then its differential $\nabla L_{\lambda}[0]$ takes the form $\nabla L_{\lambda}h=Ah+C(\lambda,h)$, where $A$ is a self-adjoint Fredholm operator and $C$ is compact (actually, it is a linear self-adjoint operator). Moreover, we have $L_{\lambda}0=0$ for all $\lambda \in \R$. Now, if we take $\lambda_{k}<\lambda_{-}<\lambda_{k+1}<\lambda_{+}$ then we can easily check that the operators $\nabla L_{\lambda_{\pm}}[0]$ are invertible. Moreover, the spectral flow of $\nabla L_{\lambda}[0]$ on $[\lambda_{-},\lambda_{+}]$ is well defined and can be computed explicitly by
\begin{align}
Sf(\nabla L_{\lambda}[0],[\lambda_{-},\lambda_{+}])&=dim(H^{-}_{\lambda_{-}}\cap H^{+}_{\lambda_{+}})-dim(H^{-}_{\lambda_{+}}\cap H^{+}_{\lambda_{-}})\notag\\
&=-dim \ker(D_{g}-\lambda_{k+1})\not=0.\notag
\end{align}
Hence, by Theorem in \cite[Theorem~1]{FPR}, we have that $0$ is a bifurcation point and hence, there exists a nontrivial solution $\phi_{\lambda}$ of $L_{\lambda}\phi_{\lambda}=0$ for $\lambda$ close to $\lambda_{k+1}$. Moreover, $\phi_{\lambda}\to 0$ as $\lambda \to \lambda_{k+1}$.

\noindent
In what follows, we will show the existence of a non-trivial ground state solution $\psi_{\lambda}$ without restriction on $\lambda>0$ as long as $\lambda \not \in Spec(D_{g})$. Moreover, if $0<\lambda \in (\lambda_{k},\lambda_{k+1})$, the solution $\psi_{\lambda}$ can be thought of as the extension of the bifurcation branch $\phi_{\lambda}$ to all the interval $(\lambda_{k},\lambda_{k+1})$. We first start by preparing the variational setting allowing the construction of a minimizing sequence. The setting is very close to the work of Sire and Xu in \cite{YX1} and the bifurcation result is close in nature to the one in \cite{BX} but we will not address the case when $\lambda \in Spec(D_{g})$, although we expect a similar result to hold in our setting.

\noindent
We recall that the energy functional of the problem has the following expression
$$J_{g,\lambda}(\psi)=:J_{\lambda}(\psi)=\frac{1}{2}\int_{M}\langle D \psi,\psi\rangle -\lambda |\psi|^{2}  dv -\frac{1}{4}\int_{M\times M} G^{s}(x,y)|\psi|^{2}(x) |\psi|^{2}(y) dv(x) dv(y).$$
For $\psi\in H^{\frac{1}{2}}$ we will write $\psi = \psi^{+}+\psi^{-} \in H_{\lambda}^{+}\oplus H_{\lambda}^{-}$, if $\lambda$ is not an eigenvalue of $D_{g}$.
\begin{proposition}\label{proptau}
There exits a $C^{1}$-map $\tau: H^{+}_{\lambda}\to H^{-}_{\lambda}$ such that for every $\psi \in H^{+}_{\lambda}$ 
$$J_{\lambda}(\psi+h)<J_{\lambda}(\psi+\tau(\psi)), \forall h\in H^{-}_{\lambda}, h\not=\tau(h).$$
Moreover, $\tau$ satisfies the following properties:
\begin{itemize}
\item[i)] $P^{-}_{\lambda}\Big[D_{g}\tau(\psi)-\Big(\int_{M}G^{s}(x,y)|\tau(\psi)+\psi|^{2}(y)\ dv(y) \Big)(\psi+\tau(\psi))\Big]=0$.
\item[ii)] $\|\tau(\psi)\|_{\lambda}^{2}\leq \frac{1}{2}\int_{M\times M}|\psi|^{2}(x)G^{s}(x,y)|\psi|^{2}(y)\ dv(y)dv(x)$.
\item[iii)] If $K(\psi):=\frac{1}{4}\int_{M\times M} G^{s}(x,y)|\psi|^{2}(x)|\psi|^{2}(y)\ dv(x)dv(y)$, then $$\|\nabla \tau(\psi)\|_{Op}\leq \|\nabla^{2}K(\psi+\tau(\psi))\|_{Op}.$$
\item[iv)] Let $\tilde{J}:H^{+}_{\lambda}\to \R$ defined by $\tilde{J}(\psi):=J_{\lambda}(\psi+\tau(\psi))$. If $(\psi_{k})_{k}$ is a (PS) sequence of $\tilde{J}$, then $(\psi_{k}+\tau(\psi_{k}))_{k}$ is a (PS)-sequence for $J_{\lambda}$ and
$$\|\nabla J_{\lambda}(\psi)\|=\|\nabla J_{\lambda}(\psi+\tau(\psi))\|, \forall \psi \in H^{+}_{\lambda}.$$
\end{itemize}
\end{proposition}

\begin{proof}
First notice that the functional
$$h\to J_{\lambda}(\psi+h)=\frac{1}{2}\|\psi\|_{\lambda}^{2}-\frac{1}{2}\|h\|^{2}_{\lambda}-\frac{1}{4}\int_{M}\int_{M}G^{s}(x,y)|\psi+h|^{2}(x)|\psi+h|^{2}(y)\ dv(x) dv(y),$$
defined on $H_{\lambda}^{-}$, is strictly concave and anti-coercive, hence it has a unique maximizer $\tau(\psi)$ and therefore $i)$ is satisfied. Now, since $\tau(\psi)$ is a maximizer of $J_{\lambda}(\psi+\cdot)$ on $H_{\lambda}^{-}$, we have that
$$J_{\lambda}(\psi+\tau(\psi))\geq J_{\lambda}(\psi).$$
It follows that
\begin{align}
\|\tau(\psi)\|_{\lambda}^{2}&\leq \frac{1}{2}\Big( \int_{M\times M}G^{s}(x,y)|\psi|^{2}(x)|\psi|^{2}(y) \ dv(x)dv(y)\notag\\
&\qquad -\int_{M\times M}G^{s}(x,y)|\psi+\tau(\psi)|^{2}(x)|\psi+\tau(\psi)|^{2}(y)\ dv(x)dv(y)\Big)\notag\\
&\leq \frac{1}{2}\int_{M\times M}G^{s}(x,y)|\psi|^{2}(x)|\psi|^{2}(y) \ dv(x)dv(y).\notag
\end{align}
and $ii)$ follows. We consider now the operator $T:=-\nabla^{2}_{h}J_{\lambda}(\psi+\cdot)[\tau(\psi)]:H^{-}_{\lambda}\to H^{-}_{\lambda}$ that can be expressed as
$$T(h)=-D_{\lambda}h+P_{\lambda}^{-}\Big(2\int_{M}G^{s}(x,y)\langle \psi +\tau(\psi),h\rangle\ dv(x) \left(\psi+\tau(\psi)\right)+\int_{M}G^{s}(x,y)|\psi+\tau(\psi)|^{2}dv(x) h\Big).$$
Notice that $T$ is positive definite and 
\begin{align}
\langle T(h),h\rangle &=\|h\|_{\lambda}^{2}+2\int_{M\times M}G^{s}(x,y)\langle \psi+\tau(\psi),h\rangle(x) \langle \psi+\tau(\psi),h\rangle(y)\ dv(x) dv(y)\notag\\
&\quad+\int_{M\times M}G^{s}(x,y)|\psi+\tau(\psi)|^{2}|h|^{2}\ dv(x)dv(y) \geq \|h\|_{\lambda}^{2}.
\end{align} 
Hence, it is invertible and
$$\|T^{-1}\|_{Op}\leq 1.$$
On the other hand, if $$L(h,\psi)=P^{-}_{\lambda}\Big[D_{\lambda}h-\int_{M}G^{s}(x,y)|h+\psi|^{2}(y)\ dv(y) (\psi+h)\Big],$$
then from $i)$, we have $$L(\tau(\psi),\psi)=0.$$
Applying the implicit function theorem yields 
$$\nabla \tau (\psi)\phi=-\Big((\nabla_{h}L)(\tau(\psi),\psi)\Big)^{-1}(\nabla_{\psi}L)(\tau(\psi),\psi)\phi, \text{ for all } \phi\in H^{+}_{\lambda}.$$
But $(\nabla_{h}L)(\tau(\psi),\psi)=T$ and $(\nabla_{\psi}L)(\tau(\psi),\psi)=\nabla^{2}K(\psi+\tau(\psi))$. Hence,
$$\|\nabla \tau(\psi)\|_{Op}\leq \|\nabla^{2}K(\psi+\tau(\psi))\|_{Op}.$$
and therefore $(iii)$ holds. We finish the proof now by differentiating $\tilde{J}$ in order to get
$$\nabla \tilde{J}(\psi)\phi=\nabla J_{\lambda}(\psi+\tau(\psi))[\phi+\nabla \tau(\psi)\phi]$$
But $\nabla \tau(\psi)\phi \in H^{-}_{\lambda}$ and $\tau(\psi)$ is a critical point of $J_{\lambda}(\psi+\cdot)$ restricted to $H^{-}_{\lambda}$. Therefore,
$$\nabla \tilde{J}(\psi)\phi=\nabla J_{\lambda}(\psi+\tau(\psi))\phi, \forall \phi \in H^{+}_{\lambda}.$$
In particular, if $(\psi)_{k}\subset H^{+}_{\lambda}$ is a (PS) sequence for $J_{\lambda}$, then $(\psi_{k}+\tau(\psi_{k}))_{k}$ is a (PS) sequence for $\tilde{J}$.
\end{proof}

\noindent
We claim, next, that $\tilde{J}$ has a mountain-pass geometry. Indeed, we have $\tilde{J}(0)=J(0)=0$. Moreover, if $\psi \in H^{+}_{\lambda}$ with $\|\psi\|_{\lambda}=1$, we have
$$\tilde{J}(t\psi)\geq J_{\lambda}(t\psi)=\frac{t^{2}}{2}-\frac{t^{4}}{4}\int_{M}G^{s}*|\psi|^{2}|\psi|^{2}\ dv.$$
Therefore, there exists $t_{0}>0$ and $\nu_{0}>0$ such that 
$$\tilde{J}(t\psi)\geq 0, \forall 0\leq t\leq t_{0} \text{ and } \tilde{J}(t_{0}\psi)\geq \nu_{0}.$$
In order to find a critical point for $\tilde{J}$ (and hence a critical point for $J_{\lambda}$), we define the min-max level $\delta_{\lambda}$ by setting
$$\delta_{\lambda}:=\inf_{\psi \in H^{+}_{\lambda}\setminus\{0\}}\max_{t>0}\tilde{J}(t\psi).$$
Notice that $\delta_{\lambda}\geq \nu_{0}>0$.
This critical level, if it exists, corresponds to the ground state of $\tilde{J}$ on the Nehari manifold
$$\mathcal{M}=\{\psi \in H^{+}_{\lambda}; \langle \nabla \tilde{J}(\psi),\psi\rangle=0\}.$$
That is, $$\delta_{\lambda}=\inf_{\psi\in \mathcal{M}} \tilde{J}(\psi),$$
as long as $\mathcal{M}\not=\emptyset$.
If $\tilde{J}$ satisfies the (PS) condition and $\mathcal{M}\not=\emptyset$, then $\delta_{\lambda}$ is indeed a critical value for $\tilde{J}$ and hence for $J_{\lambda}$. But, $\tilde{J}$ and $J_{\lambda}$ satisfy the (PS) condition only below $\overline{Y}$. So our objective now, is to show that $\delta_{\lambda}<\overline{Y}$. In the classical setting, one uses a test function (mainly grafting a standard bubble). In our case, some work needs to be done to handle the $\tau$-component of any potential test spinor. To this end, we want to be able to estimate the energy level of a (PS) sequence of $\tilde{J}$ in terms of the energy levels of $J_{\lambda}$.\\
We consider then a $(PS)_{c}$ sequence $(\psi)_{k}$ for $J_{\lambda}$. That is, $J_{\lambda}(\psi_{k})\to c>0$ and $\|\nabla J_{\lambda}(\psi_{k})\|\to 0$. Based on the study of (PS) sequences above, we know that $\|\psi_{k}\|_{\lambda}$ is bounded. Moreover, we have the following properties:
\begin{proposition}\label{prop diff}
Given a $(PS)_{c}$ sequence $(\psi_{k})_{k}$ for $J_{\lambda}$, we have
\begin{itemize}
\item[i)] $\|\psi_{k}^{-}-\tau(\psi_{k})\|_{\lambda}=O\Big(\|\nabla J_{\lambda}(\psi_{k})\|\Big)$.
\item[ii)]$\nabla\tilde{J}(\psi_{k}^{+})\to 0$.
\item[iii)] There exists $t_{k}>0$ such that $t_{k}\psi_{k}^{+}\in \mathcal{M}$. Moreover, $|t_{k}-1|=O\Big(\|\nabla \tilde{J}(\psi_{k}^{+})\|\Big)$.
\end{itemize}
\end{proposition}

\begin{proof}

We start by the proof of $i)$. We let $z_{1}=\psi_{k}^{+}+\tau(\psi_{k}^{+})$ and $z_{2}=\psi_{k}^{-}-\tau(\psi_{k}^{+})$, so that $z_{1}+z_{2}=\psi_{k}$. Recall that
$$\langle \nabla J_{\lambda}(z_{1}),z_{2}\rangle =0.$$
Therefore,
$$-\langle z_{1}^{-},z_{2}\rangle_{\lambda}-\int_{M}G^{s}*|z_{1}|^{2}\langle z_{1},z_{2}\rangle\ dv =0.$$
On the other hand, we have
$$\langle\nabla J_{\lambda}(\psi_{k}),z_{2}\rangle =-\langle \psi_{k}^{-},z_{2}\rangle_{\lambda}-\int_{M}G^{s}*|\psi_{k}|^{2}\langle \psi_{k},z_{2}\rangle \ dv.$$
Hence,
$$\langle \nabla J_{\lambda}(\psi_{k}),z_{2}\rangle =-\|z_{2}\|_{\lambda}^{2}+\int_{M}G^{s}*|z_{1}|^{2}\langle z_{1},z_{2}\rangle\ dv-\int_{M}G^{s}*|\psi_{k}|^{2}\langle \psi_{k},z_{2}\rangle \ dv.$$
Notice now, that
$$\int_{M}G^{s}*|z_{1}|^{2}\langle z_{1},z_{2}\rangle\ dv-\int_{M}G^{s}*|\psi_{k}|^{2}\langle \psi_{k},z_{2}\rangle\ dv=\langle \nabla K(z_{1}),z_{2}\rangle -\langle \nabla K(\psi_{k}),z_{2}\rangle.$$
Thus, there exists $\mu_{k}\in [0,1]$ such that
$$\int_{M}G^{s}*|z_{1}|^{2}\langle z_{1},z_{2}\rangle\ dv-\int_{M}G^{s}*|\psi_{k}|^{2}\langle \psi_{k},z_{2}\rangle=-\langle \nabla^{2}K(z_{1}+\mu_{k}z_{2})[z_{2}],z_{2}\rangle \leq 0.$$
This yields 
$$\|z_{2}\|_{\lambda}\leq \|\nabla J_{\lambda}(\psi_{k})\|\to 0,$$
as claimed in $i)$.\\ 
For the proof of $ii)$, we start with
$$\nabla J_{\lambda}(z_{1})=\nabla J_{\lambda}(\psi_{k}-z_{2}),$$ and since $z_{2}\to 0$, as claimed in $i)$, we have $\nabla J_{\lambda}(z_{1})\to 0$. Therefore,
$$\|\nabla \tilde{J}(\psi_{k}^{+})\|=\|\nabla J_{\lambda}(z_{1})\|\to 0.$$
It remains now to prove $iii)$, which is more involved. First, we claim that there exists $c_{0}>0$ such that
\begin{equation}\label{lowerb}
\int_{M}G^{s}*|z_{1}|^{2}|z_{1}|^{2}\ dv>c_{0}.
\end{equation}
Indeed, we have
$$J_{\lambda}(z_{1})-\frac{1}{2}\langle \nabla J_{\lambda}(z_{1}), z_{1} \rangle=\frac{1}{2}\int_{M}G^{s}*|z_{1}|^{2}|z_{1}|^{2} \ dv.$$
On the other hand, since $z_{2}\to 0$, we have that $J_{\lambda}(z_{1})=J_{\lambda}(\psi_{k})+o(1)=c+o(1)$ and $\langle \nabla J(z_{1}),z_{1}\rangle =o(1)$. Hence,
$$\frac{1}{2}\int_{M}G^{s}*|z_{1}|^{2}|z_{1}|^{2} \ dv=c+o(1),$$
which finishes the proof of the claim. Now we consider the function $f(t):=\langle \nabla \tilde{J}(t\psi_{k}^{+}),t\psi_{k}^{+}\rangle.$
Notice that $f(1)\to 0$ as $k\to \infty$. Moreover, we have
$$f'(1)=\langle \nabla^{2} \tilde{J}(\psi_{k}^{+})[\psi_{k}^{+}],\psi_{k}^{+}\rangle +\langle \nabla \tilde{J}(\psi_{k}^{+}),\psi_{k}^{+}\rangle.$$
But expanding the first term of the previous equation yields
\begin{align}
\langle \nabla^{2}\tilde{J}(\psi_{k}^{+})[\psi_{k}^{+}],\psi_{k}^{+}\rangle &=\|\psi_{k}^{+}\|_{\lambda}^{2}-\langle \nabla^{2}K(\psi_{k}^{+}+\tau(\psi_{k}^{+}))[\psi_{k}^{+}+\nabla \tau(\psi_{k}^{+})[\psi_{k}^{+}]],\psi_{k}^{+}\rangle\notag\\
&=\langle \nabla \tilde{J}(\psi_{k}^{+}),\psi_{k}^{+}\rangle+\langle \nabla K(\psi_{k}^{+}+\tau(\psi_{k}^{+})),\psi_{k}^{+}\rangle\notag\\
&\quad-\langle \nabla^{2}K(\psi_{k}^{+}+\tau(\psi_{k}^{+}))[\psi_{k}^{+}+\nabla \tau(\psi_{k}^{+})[\psi_{k}^{+}]],\psi_{k}^{+}\rangle.\notag
\end{align}
We set then $z_{k}:=\psi_{k}^{+}+\tau(\psi_{k}^{+})$ and $h_{k}:=\nabla \tau(\psi_{k}^{+})[\psi_{k}^{+}]-\tau(\psi_{k}^{+})$. Then we have
\begin{align}
\langle \nabla K(z_{k}),\psi_{k}^{+}\rangle-\langle \nabla^{2}K(z_{k})[z_{k}+h_{k}],\psi_{k}^{+}\rangle&=\langle \nabla K(z_{k}),z_{k}\rangle-\langle \nabla K(z_{k}),\tau(\psi_{k}^{+})\rangle\notag\\
&\quad-\langle \nabla^{2}K(z_{k})[z_{k}+h_{k}],z_{k}+h_{k}\rangle\notag\\
&\quad+\langle \nabla^{2}K(z_{k})[z_{k}+h_{k}],\nabla\tau(\psi_{k}^{+})\psi_{k}^{+}\rangle.\notag
\end{align}
On the other hand, by differentiating $i)$ in Proposition \ref{proptau} with respect to $\psi$, we have
$$-\langle \nabla \tau(\psi_{k}^{+})[\psi_{k}^{+}],w\rangle_{\lambda}=\langle \nabla^{2} K(z_{k})[z_{k}+h_{k}],w\rangle.$$
In particular,
$$-\|\nabla \tau(\psi_{k}^{+})[\psi_{k}^{+}]\|_{\lambda}^{2}=\langle \nabla^{2} K(z_{k})[z_{k}+h_{k}],\nabla \tau(\psi_{k}^{+})[\psi_{k}^{+}]\rangle.$$
Moreover, we have
$$-\|\tau(\psi_{k})^{+}\|^{2}=\langle \nabla K(z_{k}),\tau(\psi_{k}^{+})\rangle\quad \text{ and }\quad -\langle \tau(\psi_{k}^{+}),\nabla \tau(\psi_{k}^{+})[\psi_{k}^{+}]\rangle =\langle \nabla K(z_{k}), \nabla \tau(\psi_{k}^{+})[\psi_{k}^{+}]\rangle.$$
Hence,
\begin{align}
\langle \nabla K(z_{k}),\psi_{k}^{+}\rangle-\langle \nabla^{2}K(z_{k})[z_{k}+h_{k}],\psi_{k}^{+}\rangle&=-\|h_{k}\|_{\lambda}^{2}+\langle \nabla K(z_{k}),z_{k}\rangle+2\langle \nabla K(z_{k}),h_{k}\rangle \notag\\
&\quad-\langle \nabla^{2}K(z_{k})[z_{k}+h_{k}],z_{k}+h_{k}\rangle.\notag
\end{align}
Thus,
$$f'(1)=2f(1)+\langle \nabla K(z_{k}),z_{k}\rangle-\langle \nabla^{2}K(z_{k})[z_{k}+h_{k}],z_{k}+h_{k}\rangle+2\langle \nabla K(z_{k}),h_{k}\rangle -\|h_{k}\|_{\lambda}^{2}.$$
In order to evaluate the sign of $f'(1)$, we need to expand $\langle \nabla K(z_{k}),z_{k}\rangle-\langle \nabla^{2}K(z_{k})[z_{k}+h_{k}],z_{k}+h_{k}\rangle+2\langle \nabla K(z_{k}),h_{k}\rangle$. Indeed,
\begin{align}
&\langle \nabla K(z_{k}),z_{k}\rangle-\langle \nabla^{2}K(z_{k})[z_{k}+h_{k}],z_{k}+h_{k}\rangle+2\langle \nabla K(z_{k}),h_{k}\rangle=\int_{M}G^{s}*|z_{k}|^{2}|z_{k}|^{2}\ dv \notag\\
&\quad-\int_{M}G^{s}*|z_{k}|^{2}|z_{k}+h_{k}|^{2}\ dv -2\int_{M}G^{s}*\langle z_{k}+h_{k},z_{k}\rangle\langle z_{k}+h_{k},z_{k}\rangle \ dv \notag\\
&\quad+2\int_{M}G^{s}*|z_{k}|^{2}\langle z_{k},h_{k}\rangle \ dv\notag\\
&=-2\int_{M}G^{s}*\langle z_{k}+h_{k},z_{k}\rangle\langle z_{k}+h_{k},z_{k}\rangle \ dv-\int_{M}G^{s}*|z_{k}|^{2}|h_{k}|^{2}\ dv\notag\\
&\leq -2\int_{M}G^{s}*\langle z_{k}+h_{k},z_{k}\rangle\langle z_{k}+h_{k},z_{k}\rangle \ dv\leq 0.\notag
\end{align}
Therefore,
$$f'(1)\leq 2f(1)-2\int_{M}G^{s}*\langle z_{k}+h_{k},z_{k}\rangle\langle z_{k}+h_{k},z_{k}\rangle \ dv -\|h_{k}\|^{2}_{\lambda}.$$
But $f(1)\to 0$ as $k\to \infty$, which leads to two cases. Either there exists $\mu_{0}>0$ such that $\|h_{k}\|_{\lambda}^{2}\geq\mu_{0}$ for $k$ large enough, and thus for $k$ large enough
$$f'(1)\leq -\frac{\mu_{0}}{2},$$
or $\|h_{k}\|^{2}_{\lambda}\to 0$ and in that case, $$\int_{M\times M}G^{s}(x,y)\langle z_{k}+h_{k},z_{k}\rangle(x)\langle z_{k}+h_{k},z_{k}\rangle(y) \ dv(x) dv(y)=\int_{M}G^{s}*|z_{k}|^{2}|z_{k}|^{2}\ dv +o(1).$$
Now using $(\ref{lowerb})$, we have the existence of $\mu_{0}>0$ such that
$$2\int_{M}G^{s}*\langle z_{k}+h_{k},z_{k}\rangle\langle z_{k}+h_{k},z_{k}\rangle \ dv \geq \mu_{0}.$$ In conclusion, we have for $k$ large
$$f'(1)\leq -\frac{\mu_{0}}{2}.$$
In particular, $f'(t)<-\frac{\mu_{0}}{4}$ in a small neighborhood of $1$, independent of $k$, of the form $[1-\mu,1+\mu]$ for a certain $\mu>0$ and small but fixed. Using the mean value theorem, we have
$$f(1+\mu)\leq f(1)-\frac{\mu_{0}\mu}{4}<0 \text{ and } f(1-\mu)\geq f(1)+\frac{\mu_{0}\mu}{4}>0.$$
So there exists $t_{k}\in [1-\mu,1+\mu]$ such that $f(t_{k})=0$. Moreover, since $|\frac{1}{f'(t)}|\leq \frac{4}{\mu_{0}}$ for $t\in[1-\mu,1+\mu]$, we have
$$|t_{k}-1|=|f^{-1}(0)-f^{-1}(f(1))|\leq \frac{4}{\mu_{0}}|f(1)|=O\Big(\|\nabla \tilde{J}(\psi_{k}^{+})\|\Big),$$
which finishes the proof.
\end{proof}

\begin{proposition}\label{prop exp}
Assume that $(\psi_{k})_{k}$ is a $(PS)_{c}$ sequence for $J_{\lambda}$ with $c>0$. Then
$$\delta_{\lambda}\leq J_{\lambda}(\psi_{k})+O\Big(\|\nabla J_{\lambda}(\psi_{k})\|^{2}\Big).$$
In particular, if $J_{\lambda}$ satisfies the (PS) condition at the level set $\delta_{\lambda}$ then it has a critical point $\psi$ at that level.
\end{proposition}

\begin{proof}
We will be using the notations of the previous proof. That is, we let $z_{k}=\psi_{k}^{+}+\tau(\psi_{k}^{+})$ and $w_{k}=t_{k}\psi_{k}^{+}+\tau(t_{k}\psi_{k}^{+})$. Then we have from Proposition \ref{prop diff}:
$$\|\psi_{k}-w_{k}\|_{\lambda}\leq\|\psi_{k}-z_{k}\|_{\lambda} +|t_{k}-1|\|\psi_{k}^{+}\|_{\lambda}+\|\tau(\psi_{k}^{+})-\tau(t_{k}\psi_{k}^{+})\|_{\lambda}=O\Big(\|\nabla J_{\lambda}(\psi_{k})\|\Big)+O\Big(\|\nabla \tilde{J}(\psi_{k}^{+})\|\Big).$$
On the other hand,
$$\|\nabla \tilde{J}(\psi_{k}^{+})\|=\|\nabla J_{\lambda}(z_{k})\|=\|\nabla J_{\lambda}(\psi_{k})\|+O\Big(\|z_{k}-\psi_{k}\|_{\lambda}\Big)=O\Big(\|\nabla J_{\lambda}(\psi_{k})\|\Big).$$
In particular, we have 
$$\|\psi_{k}-w_{k}\|_{\lambda}\leq O\Big(\|\nabla J_{\lambda}(\psi_{k})\|\Big).$$
Next, we notice that since $t_{k}\psi_{k}^{+}\in \mathcal{M}$, we have that
$$\langle \nabla J_{\lambda}(w_{k}),\psi_{k}-w_{k}\rangle = \langle \nabla J_{\lambda}(w_{k}), (\psi_{k}-w_{k})^{+}\rangle=(1-t_{k})\langle \nabla \tilde{J}(t_{k}\psi_{k}^{+}),\psi_{k}^{+}\rangle =0.$$
Hence,
\begin{align}
J_{\lambda}(\psi_{k})&=J_{\lambda}(w_{k})+\langle \nabla J_{\lambda}(w_{k}), \psi_{k}-w_{k}\rangle +O\Big(\|\psi_{k}-w_{k}\|^{2}_{\lambda}\Big)\notag\\
&= J_{\lambda}(w_{k})+O\Big(\|\nabla J_{\lambda}(\psi_{k})\|^{2}\Big).
\end{align}
Therefore,
$$\delta_{\lambda}\leq \tilde{J}(t_{k}\psi_{k}^{+})=J_{\lambda}(w_{k})=J_{\lambda}(\psi_{k})+O\Big(\|\nabla J_{\lambda}(\psi_{k})\|^{2}\Big).$$
\end{proof}

\subsection{Test Spinor}

\noindent
We are now ready to construct a test spinor that will allow us to go under the critical energy threshold and hence have compactness of the minimizing Palais-Smale sequence. We will closely follow the construction in \cite{I} and \cite{YX1}. 

\noindent
Consider a constant spinor $\psi_{0}$ on $\R^{n}$ such that $|\psi_{0}|^{2}=a_{n}$, where $a_{n}$ is  a constant satisfying
$$a_{n}^{\frac{n}{n-1}}\omega_{n}=2^{n}\overline{Y}c_{n}^{-\frac{1}{n-1}}.$$
Here, $c_{n}$ is the constant introduced in $(\ref{equality})$. We define now the spinor $$\Psi=\Big(\frac{1}{1+|x|^{2}}\Big)^{\frac{n}{2}}(1-x)\cdot \psi_{0},$$  so that if $f(r)=\frac{1}{1+r^{2}}$, then $|\Psi|^{2}=a_{n}f(|x|)^{n-1}$. Notice that 
$$D_{\R^{n}}\Psi=\frac{n}{2}f\Psi.$$
We fix $\delta>0$ so that $2\delta<i(M)$, the injectivity radius of $M$. We let $\eta$ to be a smooth function on $\R^{n}$ with support in $B_{2\delta}(0)=:B_{2\delta}$ such that $\eta=1$ on $B_{\delta}(0)=:B_{\delta}$. Now, we can define the spinor $\psi_{\varepsilon}(x)=\eta(x)\varepsilon^{-\frac{n-1}{2}}\Psi(\frac{x}{\varepsilon})=\eta(x) \Psi_{\varepsilon}(x)$. Next, we use the Bourguignon-Gauduchon \cite{Bour} trivialization in order to graft the spinor $\psi_{\varepsilon}$ on $M$. Indeed,  we fix $p_{0}\in M$ and $(x_{1},\cdots,x_{n})$ local normal coordinates around $p_{0}$ provided by the exponential map $\exp_{p_{0}}$. That is, there exists a neighborhood $U\subset T_{p_{0}}M=\R^{n}$ and a neighborhood $V\subset M$, such that $\exp_{p_{0}}:U\to V$ is a diffeomorphism.

\noindent
Let $G(p)=(g_{ij}(p))_{ij}$ be the components of the metric at $p$ and $B=G^{-\frac{1}{2}}$. Notice that $B$ is well defined since $G$ is symmetric and positive definite. With these notations, we have that $B^{*}g=g_{\R^{n}}$. Therefore, $B$ defines an isometry as a map $B(p):(T_{\exp^{-1}_{p_{0}}p}U,g_{\R^{n}})\to (T_{p}V,g(p))$. Hence, given an oriented frame $(y_{1},\cdots, y_{n})$ on $U$, we obtain a natural oriented frame on $V$ by taking $(By_{1},\cdots, B y_{n})$. Thus, one has an isomorphism of the $SO(n)$-principal bundle induced by the map $\Phi(y_{1},\cdots, y_{n})=(By_{1},\cdots, B y_{n})$ as described in the diagram below:

\begin{center}

\tikzset{every picture/.style={line width=0.75pt}} 

\begin{tikzpicture}[x=0.75pt,y=0.75pt,yscale=-1,xscale=1]

\draw    (211,120) -- (376,120.99) ;
\draw [shift={(378,121)}, rotate = 180.34] [color={rgb, 255:red, 0; green, 0; blue, 0 }  ][line width=0.75]    (10.93,-3.29) .. controls (6.95,-1.4) and (3.31,-0.3) .. (0,0) .. controls (3.31,0.3) and (6.95,1.4) .. (10.93,3.29)   ;
\draw    (212,212) -- (377,212.99) ;
\draw [shift={(379,213)}, rotate = 180.34] [color={rgb, 255:red, 0; green, 0; blue, 0 }  ][line width=0.75]    (10.93,-3.29) .. controls (6.95,-1.4) and (3.31,-0.3) .. (0,0) .. controls (3.31,0.3) and (6.95,1.4) .. (10.93,3.29)   ;
\draw    (182,133.5) -- (182.97,193.5) ;
\draw [shift={(183,195.5)}, rotate = 269.08] [color={rgb, 255:red, 0; green, 0; blue, 0 }  ][line width=0.75]    (10.93,-3.29) .. controls (6.95,-1.4) and (3.31,-0.3) .. (0,0) .. controls (3.31,0.3) and (6.95,1.4) .. (10.93,3.29)   ;
\draw    (395,134.5) -- (395.97,194.5) ;
\draw [shift={(396,196.5)}, rotate = 269.08] [color={rgb, 255:red, 0; green, 0; blue, 0 }  ][line width=0.75]    (10.93,-3.29) .. controls (6.95,-1.4) and (3.31,-0.3) .. (0,0) .. controls (3.31,0.3) and (6.95,1.4) .. (10.93,3.29)   ;

\draw (117,102) node [anchor=north west][inner sep=0.75pt]    {$P_{SO}( U,g_{\mathbb{R}^{n}})$};
\draw (384,102) node [anchor=north west][inner sep=0.75pt]    {$P_{SO}( V,g) \subset P_{SO}( M,g)$};
\draw (127,200) node [anchor=north west][inner sep=0.75pt]    {$U\subset T_{p_{0}} M$};
\draw (385,202) node [anchor=north west][inner sep=0.75pt]    {$V\subset M$};
\draw (281,94) node [anchor=north west][inner sep=0.75pt]    {$\Phi $};
\draw (268,181) node [anchor=north west][inner sep=0.75pt]    {$\exp_{p_{0}}$};

\end{tikzpicture}

\end{center}

\noindent
The map $\Phi$ commutes with the right action of $SO(n)$ and hence it induces an isomorphism of spin structures:

\begin{center}
\tikzset{every picture/.style={line width=0.75pt}} 

\begin{tikzpicture}[x=0.75pt,y=0.75pt,yscale=-1,xscale=1]

\draw    (211,120) -- (376,120.99) ;
\draw [shift={(378,121)}, rotate = 180.34] [color={rgb, 255:red, 0; green, 0; blue, 0 }  ][line width=0.75]    (10.93,-3.29) .. controls (6.95,-1.4) and (3.31,-0.3) .. (0,0) .. controls (3.31,0.3) and (6.95,1.4) .. (10.93,3.29)   ;
\draw    (163,211) -- (422,211.99) ;
\draw [shift={(424,212)}, rotate = 180.22] [color={rgb, 255:red, 0; green, 0; blue, 0 }  ][line width=0.75]    (10.93,-3.29) .. controls (6.95,-1.4) and (3.31,-0.3) .. (0,0) .. controls (3.31,0.3) and (6.95,1.4) .. (10.93,3.29)   ;
\draw    (105,128.5) -- (105.97,188.5) ;
\draw [shift={(106,190.5)}, rotate = 269.08] [color={rgb, 255:red, 0; green, 0; blue, 0 }  ][line width=0.75]    (10.93,-3.29) .. controls (6.95,-1.4) and (3.31,-0.3) .. (0,0) .. controls (3.31,0.3) and (6.95,1.4) .. (10.93,3.29)   ;
\draw    (469,132.5) -- (469.97,192.5) ;
\draw [shift={(470,194.5)}, rotate = 269.08] [color={rgb, 255:red, 0; green, 0; blue, 0 }  ][line width=0.75]    (10.93,-3.29) .. controls (6.95,-1.4) and (3.31,-0.3) .. (0,0) .. controls (3.31,0.3) and (6.95,1.4) .. (10.93,3.29)   ;

\draw (7,102) node [anchor=north west][inner sep=0.75pt]    {$U\times Spin( n) =P_{Spin}( U,g_{\mathbb{R}^{n}})$};
\draw (384,102) node [anchor=north west][inner sep=0.75pt]    {$P_{Spin}( V,g) \subset P_{Spin}( M,g)$};
\draw (75,200) node [anchor=north west][inner sep=0.75pt]    {$U\subset T_{p_{0}} M$};
\draw (431,199) node [anchor=north west][inner sep=0.75pt]    {$V\subset M$};
\draw (275,89) node [anchor=north west][inner sep=0.75pt]    {$\tilde{\Phi }$};
\draw (268,181) node [anchor=north west][inner sep=0.75pt]    {$\exp_{p_{0}}$};

\end{tikzpicture}
\end{center}

\noindent
This leads to an isomorphism between the spin bundles $\Sigma_{g_{\R^{n}}} U$ and $\Sigma_{g} V$. If we let $e_{i}=B(\partial_{x_{i}})$ we then obtain an orthonormal frame $(e_{1},\cdots e_{n})$ of $(TV,g)$. We let $\nabla$ and $\overline{\nabla}$, respectively the Levi-Civita connections on $(TU,g_{\R^{n}})$ and $(TV,g)$. We will keep the same notations for their natural lifts to $\Sigma_{g_{\R^{n}}} U$ and $\Sigma_{g} V$. From now on, if $H\to U$ (resp. $H\to V$) is a smooth bundle over $U$ (resp. over $V$), we let $\Gamma(H)$ be the space of smooth sections of $H$. The Clifford multiplications then satisfy
$$e_{i}\cdot \overline{\psi}=B(\partial_{x_{i}})\cdot \overline{\psi}=\overline{\partial_{x_{i}}\cdot \psi},$$
where here we use the identification that any $\psi \in \Gamma (\Sigma_{g_{\R^{n}}} U)$ corresponds via the previously defined isomorphism to a spinor $\overline{\psi}\in \Gamma(\Sigma_{g} V)$. If $D$ and $\overline{D}$ are the Dirac operators acting on $\Gamma (\Sigma_{g_{\R^{n}}} U)$ and $\Gamma (\Sigma_{g} V)$, then we have for $\psi \in \Gamma(\Sigma U)$ 
$$\bar{D} \overline{\psi}=\overline{D\psi} +W\cdot \overline{\psi}+X\cdot \overline{\psi}+\sum_{i,j}(b_{ij}-\delta_{ij})\overline{\partial_{x_{i}}\cdot \nabla_{\partial_{x_{j}}}\psi},$$
where here, the $b_{ij}$ are such that $e_{i}=\sum_{j} b_{ij}\partial_{x_{j}}$, $W\in \Gamma(Cl(TV))$ and $X\in \Gamma(TV)$ are defined by
$$W=\frac{1}{4} \sum_{i,j,k;\\ i\not=j\not=k\not=i} \sum_{\alpha,\beta} b_{i\alpha}(\partial_{x_{\alpha}}b_{j\beta})b^{-1}_{\beta k}e_{i}\cdot e_{j}\cdot e_{k},$$
and
$$X=\frac{1}{4}\sum_{i,k}(\overline{\Gamma}_{ik}^{i}-\overline{\Gamma}_{ii}^{k})e_{k}=\frac{1}{2}\sum_{i,k}\overline{\Gamma}_{ik}^{i}e_{k}.$$
Using the identification between $x\in \R^{n}$ and $p=\exp_{p_{0}}x\in M$, we can write as in \cite{I,YX1}, that $G=I+O(|x|^{2})$ as $|x|\to 0$. Hence, we have
$$b_{ij}=\delta_{ij}+O(|x|^{2}), \quad W=O(|x|^{3}) \text{ and } X=O(|x|) \text{ as } |x|\to 0.$$
Our test spinor then, will be $\varphi_{\varepsilon}:=\overline{\psi_{\varepsilon}}$. Our ultimate goal in here is to apply Proposition \ref{prop exp} for the test spinor $\varphi_{\varepsilon}$. In order to do that, we need to show that $(\varphi_{\varepsilon})_{\varepsilon}$ is a $(PS)_{c}$ sequence for $J_{\lambda}$. So we start by estimating the gradient of $J_{\lambda}$ at $\varphi_{\varepsilon}$: 
\begin{lemma}\label{lemma1 exp}
For $\varphi_{\varepsilon}$ defined as above, we have
$$\|\nabla J_{\lambda}(\varphi_{\varepsilon})\|_{H_{\lambda}^{*}}\leq \left\{\begin{array}{ll}
O(\varepsilon |\ln(\varepsilon)|^{\frac{2}{3}}), \text{ if } n=3\\
\\
\varepsilon, \text{ if } n\geq 4
\end{array}
\right. .$$
\end{lemma}

\begin{proof}
We need to estimate the $H_{\lambda}^{*}$-norm of $\varphi_{\varepsilon}$ and $R_{\varepsilon}=\overline{D}\varphi_{\varepsilon}-\Big(G_{s}*|\varphi_{\varepsilon}|^{2}\Big)\varphi_{\varepsilon}$, where $H^{*}_{\lambda}$ here is the dual of the space $H^{\frac{1}{2}}(M)$ equipped with the norm $\|\cdot \|_{\lambda}$. One notices that since $\lambda \not \in Spec(D_{g})$, the $\|\cdot \|_{\lambda}$- norm is equivalent to the usual $H^{\frac{1}{2}}(M)$ norm. Hence, by the continuous embedding $L^{\frac{2n}{n+1}}(M)\hookrightarrow H^{-\frac{1}{2}}(M)$, we have that for all $\psi \in L^{\frac{2n}{n+1}}(M)$, 
$$\|\psi\|_{H_{\lambda}^{*}}\leq C\|\psi\|_{L^{\frac{2n}{n+1}}}.$$
Therefore we have
\begin{align}
\|\varphi_{\varepsilon}\|_{H^{*}_{\lambda}}&\leq C \|\varphi_{\varepsilon}\|_{L^{\frac{2n}{n+1}}}=\Big(\int_{B_{2\delta}}|\varphi_{\varepsilon}|^{\frac{2n}{n+1}}\ dv\Big)^{\frac{n+1}{2n}}\notag\\
&\leq C \Big(\int_{|x|\leq 2\delta} |\psi_{\varepsilon}|^{\frac{2n}{n+1}}\ dx \Big)^{\frac{n+1}{2n}}\notag\\
&\leq C\varepsilon\Big(\int_{0}^{\frac{2\delta}{\varepsilon}} \frac{r^{n-1}}{(1+r^{2})^{\frac{n(n-1)}{n+1}}}\ dr \Big)^{\frac{n+1}{2n}} \notag\\
&\leq C\left\{\begin{array}{ll}
\varepsilon |\ln(\varepsilon)|^{\frac{2}{3}} \text{ if } n=3\\
\varepsilon \text{ if } n\geq 4 
\end{array}
\right..\label{phieps}
\end{align}
Next, we move to estimating $R_{\varepsilon}$. Indeed, we have
\begin{align}
\overline{D}\varphi_{\varepsilon}&=\overline{D\psi_{\varepsilon}}+W\cdot \overline{\psi_{\varepsilon}}+X\cdot \overline{\psi}+\sum_{i,j}(b_{ij}-\delta_{ij})\overline{\partial_{x_{i}}\cdot \nabla_{\partial_{x_{j}}}\psi_{\varepsilon}}\notag\\
&=\Big(\int_{\R^{n}}G_{\R^{n}}^{s}(x,y)|\Psi_{\varepsilon}(y)|^{2}\ dy\Big) \varphi_{\varepsilon}+(\nabla\eta(x)+X) \cdot \varphi_{\varepsilon}+W\cdot \varphi_{\varepsilon}\notag\\
&\quad+\sum_{i,j}(b_{ij}-\delta_{ij})\overline{\partial_{x_{i}}\cdot \nabla_{\partial_{x_{j}}}\psi_{\varepsilon}}.\label{decomp}
\end{align}
On the other hand,
\begin{align}
\overline{D \psi_{\varepsilon}}-(G^{s}_{g}*|\varphi_{\varepsilon}|^{2})\varphi_{\varepsilon}&=\Big(\int_{\R^{n}}G_{\R^{n}}^{s}(x,y)|\Psi_{\varepsilon}(y)|^{2}\ dy \Big)\varphi_{\varepsilon}-\Big(\int_{M}G_{g}^{s}(x,y)|\varphi_{\varepsilon}|^{2}\ dy\Big) \varphi_{\varepsilon}+\nabla\eta(x) \cdot \varphi_{\varepsilon}\notag\\
&=\Big(\int_{|x-y|<\frac{\delta}{2}}[G_{\R^{n}}^{s}(x,y)-G_{g}^{s}(x,y)]|\varphi_{\varepsilon}(y)|^{2}\ dy \Big) \varphi_{\varepsilon}\notag\\
&\quad+\Big(\int_{|x-y|>\frac{\delta}{2}}G_{\R^{n}}^{s}(x,y)|\Psi_{\varepsilon}|^{2} dy\Big) \varphi_{\varepsilon}-\Big(\int_{|x-y|>\frac{\delta}{2}}G_{g}^{s}(x,y)|\varphi_{\varepsilon}|^{2}\ dy \Big) \varphi_{\varepsilon}\notag\\
&\quad +\nabla \eta \cdot \varphi_{\varepsilon}.\notag
\end{align}
This leads to 
\begin{align}
\overline{D}\varphi_{\varepsilon}-(G_{g}^{s}*|\varphi_{\varepsilon}|^{2})\varphi_{\varepsilon}&=\Big(\int_{|x-y|<\frac{\delta}{2}}[G_{\R^{n}}^{s}(x,y)-G_{g}^{s}(x,y)]|\varphi_{\varepsilon}(y)|^{2}\ dy \Big) \varphi_{\varepsilon}\notag\\
&\quad+\Big(\int_{|x-y|>\frac{\delta}{2}}G_{\R^{n}}^{s}(x,y)|\Psi_{\varepsilon}|^{2} dy\Big) \varphi_{\varepsilon}-\Big(\int_{|x-y|>\frac{\delta}{2}}G_{g}^{s}(x,y)|\varphi_{\varepsilon}|^{2}\ dy \Big) \varphi_{\varepsilon} \notag\\
&\quad +\nabla \eta \cdot \varphi_{\varepsilon} + W\cdot \overline{\psi_{\varepsilon}}+X\cdot \overline{\psi}+\sum_{i,j}(b_{ij}-\delta_{ij})\overline{\partial_{x_{i}}\cdot \nabla_{\partial_{x_{j}}}\psi_{\varepsilon}}\notag\\
&=A_{1}+A_{2}+A_{3}+A_{4}+A_{5}+A_{6}+A_{7}.\notag
\end{align}
We will estimate now the terms $A_{i}, i=1,\cdots,7$. Indeed, for $A_{4}$, we have
\begin{align}
\|A_{4}\|_{H^{*}_{\lambda}}&\leq C\|A_{4}\|_{L^{\frac{2n}{n+1}}}=C\Big(\int_{B_{2\delta}}|\overline{\nabla \eta \cdot \Psi_{\varepsilon}}|^{\frac{2n}{n+1}}\ dv \Big)^{\frac{n+1}{2n}}\notag\\
&\leq C\Big(\int_{\delta\leq |x|\leq 2\delta}|\Psi_{\varepsilon}|^{\frac{2n}{n+1}}\ dx \Big)^{\frac{n+1}{2n}}\notag\\
&\leq C \varepsilon \Big(\int_{\frac{\delta}{\varepsilon}}^{\frac{2\delta}{\varepsilon}}\frac{r^{n-1}}{(1+r^{2})^{\frac{n(n-1)}{(n+1)}}}\ dr \Big)^{\frac{n+1}{2n}}\leq C \varepsilon^{\frac{n-1}{2}}.\notag
\end{align}
For $A_{1}$, we use Proposition \ref{propgreen}  in Section 2 in order to have
\begin{align}
\|A_{1}\|_{H^{*}_{\lambda}}&\leq C\|A_{1}\|_{L^{2n}{n+1}}\leq C\Big(\int_{B_{2\delta}}\Big(\int_{B_{2\delta}}\frac{1}{|x-y|}|\Psi_{\varepsilon}(y)|^{2}\ dy |\Psi_{\varepsilon}(x)|\ \Big)^{\frac{2n}{n+1}}dx\Big)^{\frac{n+1}{2n}}\notag\\
&\leq C \|\Psi_{\varepsilon}\|_{L^{\frac{2n}{n-1}}}\int_{B_{2\delta}} \Big(\int_{B_{2\delta}}\frac{1}{|x-y|}|\Psi_{\varepsilon}(y)|^{2}\ dy\Big)^{n} \ dx \Big)^{\frac{1}{n}}\notag\\
&\leq C\|\Psi_{\varepsilon}\|_{L^{\frac{2n}{n+1}}(B_{2\delta})}^{2}\notag\\
&\leq  C\left\{\begin{array}{ll}
\varepsilon^{2} |\ln(\varepsilon)|^{\frac{4}{3}} \text{ if } n=3\\
\varepsilon^{2} \text{ if } n\geq 4
\end{array}
\right..\notag
\end{align}
For $A_{2}$, we use the fact that the Green's function is bounded outside of the diagonal. Thus, 
\begin{align}
\|A_{2}\|_{H^{*}_{\lambda}}&\leq C \|\Psi_{\varepsilon}\|_{L^{2}}^{2}\|\varphi_{\varepsilon}\|_{L^{\frac{2n}{n+1}}}\notag\\
& \leq C \varepsilon^{2}.\notag
\end{align}
A similar inequality holds for $\|A_{3}\|_{H^{*}_{\lambda}}$. On the other hand,
\begin{align}
\|A_{5}\|_{H^{*}_{\lambda}}&\leq C \Big(\int_{B_{2\delta}}|W|^{\frac{2n}{n+1}}|\varphi_{\varepsilon}|^{\frac{2n}{n+1}}\ dv \Big)^{\frac{n+1}{n}} \leq C \Big( \int_{|x|\leq 2\delta} |x|^{\frac{6n}{n+1}}|\Psi_{\varepsilon}|^{\frac{2n}{n+1}}\ dx\Big)^{\frac{n+1}{2n}} \notag\\
&\leq C \varepsilon^{4} \int_{0}^{\frac{2\delta}{\varepsilon}} \frac{r^{\frac{6n}{n+1}+n-1}}{(1+r^{2})^{\frac{n(n-1)}{n+1}}}\ dr \Big)^{\frac{n+1}{2n}}\notag\\
&\leq C \left\{\begin{array}{ll}
\varepsilon^{\frac{n-1}{2}} \text{ if } 3\leq n\leq 8\\
\varepsilon^{4}|\ln(\varepsilon)|^{\frac{5}{9}} \text{ if } n=9\\
\varepsilon^{4} \text{ if } n\geq 10
\end{array}
\right..\notag
\end{align}
Similarly for $A_{6}$ we have
\begin{align}
\|A_{6}\|_{H^{*}_{\lambda}}&\leq C \Big(\int_{B_{2\delta}}|X|^{\frac{2n}{n+1}}|\varphi_{\varepsilon}|^{\frac{2n}{n+1}}\ dv \Big)^{\frac{n+1}{2n}}\notag\\
&\leq C \Big(\int_{|x|\leq 2\delta}|x|^{\frac{2n}{n+1}}|\psi_{\varepsilon}|^{\frac{2n}{n+1}}\ dx \Big)^{\frac{n+1}{2n}}\notag\\
&\leq C  \varepsilon^{2} \Big(\int_{0}^{\frac{2\delta}{\varepsilon}} \frac{r^{\frac{2n}{n+1}+m-1}}{(1+r^{2})^{\frac{n(n-1)}{n+1}}}\ dx \Big)^{\frac{n+1}{2n}}\notag\\
&\leq C \left\{\begin{array}{ll}
\varepsilon^{\frac{n-1}{2}} \text{ if } n=3,4\\
\varepsilon^{2}|\ln(\varepsilon)|^{\frac{3}{5}} \text{ if } n=5\\
\varepsilon^{2} \text{ if } n\geq 6
\end{array}
\right..\notag
\end{align}
It remains now to estimate $A_{7}$. We will write $A_{7}=B_{1}+B_{2}$, where 
$$B_{1}:=\eta\sum_{i,j}(b_{ij}-\delta_{ij})\overline{\partial_{x_{i}}\cdot \nabla_{\partial_{x_{j}}}\Psi_{\varepsilon}}\quad \text{ and  }\quad B_{2}:=\sum_{i,j}(b_{ij}-\delta_{ij})(\partial_{x_{j}}\eta)\overline{\partial_{x_{i}}\cdot \Psi_{\varepsilon}}.$$
Notice that since $|\nabla \Psi|\leq C f(r)^{\frac{n}{2}}$, we have
\begin{align}
\|B_{1}\|_{H^{*}_{\lambda}}&\leq C \Big(\int_{|x|\leq 2\delta}|x|^{\frac{4n}{n+1}}|\nabla \Psi_{\varepsilon}|^{\frac{2n}{n+1}}\ dx \Big)^{\frac{n+1}{2n}}\notag\\
& \leq C \varepsilon^{2}\Big(\int_{0}^{\frac{2\delta}{\varepsilon}}\frac{r^{\frac{4n}{n+1}+n-1}}{(1+r^{2})^{\frac{n^{2}}{n+1}}}\ dr \Big)^{\frac{n+1}{2n}}\notag\\
&\leq \left\{\begin{array}{ll}
\varepsilon^{\frac{n-1}{2}} \text{ if } n=3,4\\
\varepsilon^{2}|\ln(\varepsilon)|^{\frac{3}{5}} \text{ if } n=5\\
\varepsilon^{2} \text{ if } n\geq 6
\end{array}
\right..\notag
\end{align}
We finish now by estimating $B_{2}$: 
\begin{align}
\|B_{2}\|_{H^{*}_{\lambda}}&\leq C \Big(\int_{|x|\leq 2\delta}|x|^{\frac{6n}{n+1}}|\Psi_{\varepsilon}|^{\frac{2n}{n+1}}\ dx \Big)^{\frac{n+1}{2n}}\notag\\
&\leq C  \left\{\begin{array}{ll}
\varepsilon^{\frac{n-1}{2}} \text{ if } 3\leq n\leq 8\\
\varepsilon^{4}|\ln(\varepsilon)|^{\frac{5}{9}} \text{ if } n=9\\
\varepsilon^{4} \text{ if } n\geq 10
\end{array}
\right..\notag
\end{align}
All the previous estimates can be summarized as follows:
\begin{equation}\label{Reps}
\|R_{\varepsilon}\|_{H^{*}_{\lambda}}\leq C\left\{\begin{array}{ll}
\varepsilon^{\frac{n-1}{2}} \text{ if } n=3,4\\
\varepsilon^{2}|\ln(\varepsilon)|^{\frac{3}{5}} \text{ if } n=5\\
\varepsilon^{2} \text{ if } n\geq 6
\end{array}
\right..
\end{equation}
Combining $(\ref{phieps})$ and $(\ref{Reps})$ yields the desired result.
\end{proof}

\noindent
After observing that $(\varphi_{\varepsilon})_{\varepsilon}$ is indeed a $(PS)$ sequence with a precise estimate on $\|\nabla J_{\lambda}(\varphi_{\varepsilon})\|_{H^{*}_{\lambda}}$, we proceed now to estimate the energy.
\begin{lemma}\label{lemma 2 exp}
For $\varphi_{\varepsilon}$ defined as above, we have
\begin{itemize}
\item[i)] $\|\varphi_{\varepsilon}\|_{L^{2}}=Q(\varepsilon)+ C\left\{\begin{array}{ll}
O(\varepsilon^{n-1}) \text{ if } n=3\\
O(\varepsilon^{3}|\ln(\varepsilon)|) \text{ if } n=4\\
O(\varepsilon^{3}) \text{ if } n\geq 5
\end{array}
\right.
$,\\ where $ Q(\varepsilon)=\varepsilon a_{n} \omega_{n-1}\int_{0}^{\infty} \frac{r^{n-1}}{(1+r^{2})^{n-1}}\ dr$.
\item[ii)] $J_{g}(\varphi_{\varepsilon})\leq \overline{Y}+O(\varepsilon^{2}).$
\end{itemize}
\end{lemma}

\begin{proof}
Recall  that the volume form in normal coordinates takes the form $dv_{g}=dx+O(|x|^{2})$ around $p_{0}$. Hence we have
\begin{align}
\int_{M}|\varphi_{\varepsilon}|^{2}\ dv_{g}&= \int_{B_{2\delta}}|\varphi_{\varepsilon}|^{2}\ dv_{g}\notag\\
&=\int_{|x|\leq \delta}|\Psi_{\varepsilon}|^{2}\ dx + \int_{\delta\leq |x|\leq 2\delta}|\eta(x)\Psi_{\varepsilon}|^{2}\ dx +O\Big( \int_{|x|\leq \delta}|x|^{2}|\Psi_{\varepsilon}|^{2}\ dx\Big)\notag\\
&=\varepsilon a_{n}\omega_{n-1}\int_{0}^{\frac{\delta}{\varepsilon}}\frac{r^{n-1}}{(1+r^{2})^{n-1}}\ dr +O\Big(\varepsilon \int_{\frac{\delta}{\varepsilon}}^{\frac{2\delta}{\varepsilon}}\frac{r^{n-1}}{(1+r^{2})^{n-1}}\ dr\Big)+O\Big(\varepsilon^{3}\int_{0}^{\frac{2\delta}{\varepsilon}} \frac{r^{n+1}}{(1+r^{2})^{n-1}}\ dr \Big)\notag\\
&=Q(\varepsilon)+O(\varepsilon^{n-1})+ C\left\{\begin{array}{ll}
O(\varepsilon^{n-1}) \text{ if } n=3\\
O(\varepsilon^{3}|\ln(\varepsilon)|) \text{ if } n=4\\
O(\varepsilon^{3}) \text{ if } n\geq 5
\end{array}
\right.\notag\\
&=Q(\varepsilon)+ \left\{\begin{array}{ll}
O(\varepsilon^{n-1}) \text{ if } n=3\\
O(\varepsilon^{3}|\ln(\varepsilon)|) \text{ if } n=4\\
O(\varepsilon^{3}) \text{ if } n\geq 5
\end{array}
\right..\label{L2norm}
\end{align}
Here, $$Q(\varepsilon):=\varepsilon a_{n} \omega_{n-1}\int_{0}^{\infty} \frac{r^{n-1}}{(1+r^{2})^{n-1}}\ dr.$$
Next, we estimate $\int_{M}\langle \overline{D} \varphi_{\varepsilon},\varphi_{\varepsilon}\rangle \ dv_{g}$. Using the same decomposition as in $(\ref{decomp})$, we see that
\begin{align}
\int_{M}\langle \overline{D} \varphi_{\varepsilon},\varphi_{\varepsilon}\rangle \ dv_{g}&= \int_{M}\int_{\R^{n}}G_{\R^{n}}^{s}(x,y)|\Psi_{\varepsilon}|^{2}(y) |\varphi_{\varepsilon}|^{2}(x) \ dy \ dv_{g}(x)+\int_{M}\eta^{2} \langle W\cdot \overline{\Psi_{\varepsilon}},\overline{\Psi_{\varepsilon}}\rangle \ dv_{g}\notag\\
&\quad + \int_{M}\sum_{i,j}(b_{ij}-\delta_{ij}) \eta^{2} \langle \overline{\partial_{x_{i}}\cdot \nabla_{\partial_{x_{j}}}\Psi_{\varepsilon}}, \overline{\Psi_{\varepsilon}}\rangle \ dv_{g}\notag\\
&=F_{1}+F_{2}+F_{3}.\notag
\end{align}
We will estimate each term individually starting by $F_{1}$. Indeed,
\begin{align}
F_{1}&=\int_{M}\int_{\R^{n}}G_{\R^{n}}^{s}(x,y)|\Psi_{\varepsilon}|^{2}(y) \eta (x)|\overline{\Psi_{\varepsilon}}|^{2}(x) \ dy \ dv_{g}(x)\notag\\
&=\int_{|x|\leq \delta}\int_{\R^{n}}G_{\R^{n}}^{s}(x,y)|\Psi_{\varepsilon}|^{2}(y) |\Psi_{\varepsilon}|^{2}(x) \ dy dx \notag\\
&\quad+O\Big(\int_{\delta\leq |x|\leq 2\delta}\int_{\R^{n}}G_{\R^{n}}^{s}(x,y)|\Psi_{\varepsilon}|^{2}(y) |\Psi_{\varepsilon}|^{2}(x) \ dy dx \Big)\notag\\
&\quad+O\Big(\int_{|x|\leq 2\delta}\int_{\R^{n}}G_{\R^{n}}^{s}(x,y)|\Psi_{\varepsilon}|^{2}(y) |x|^{2}|\Psi_{\varepsilon}|^{2}(x) \ dy dx\Big).\notag
\end{align}
But recall that
$$\int_{\R^{n}}G_{\R^{n}}^{s}(x,y)|\Psi_{\varepsilon}|^{2}(y)\ dy=c_{n}^{\frac{1}{n-1}}|\Psi_{\varepsilon}|^{\frac{2}{n-1}}(x),$$
where $c_{n}$ is the constant defined in $(\ref{equality})$. Hence,
\begin{align}
\int_{|x|\leq \delta}\int_{\R^{n}}G_{\R^{n}}^{s}(x,y)|\Psi_{\varepsilon}|^{2}(y) |\Psi_{\varepsilon}|^{2}(x) \ dy dx&=c_{n}^{\frac{1}{n-1}}a_{n}^{\frac{n}{n-1}}\omega_{n-1}\int_{0}^{\frac{\delta}{\varepsilon}}\frac{r^{n-1}}{(1+r^{2})^{n}}\ dr\notag\\
&=c_{n}^{\frac{1}{n-1}}a_{n}^{\frac{n}{n-1}}\omega_{n-1}\int_{0}^{+\infty}\frac{r^{n-1}}{(1+r^{2})^{n}}\ dr +O(\varepsilon^{n}).\label{Gpsi}
\end{align}
On the other hand, 
\begin{align}
\int_{\delta\leq |x|\leq 2\delta}\int_{\R^{n}}G_{\R^{n}}^{s}(x,y)|\Psi_{\varepsilon}|^{2}(y) |\Psi_{\varepsilon}|^{2}(x) \ dy dx &=c_{n}^{\frac{1}{n-1}}a_{n}^{\frac{n}{n-1}}\omega_{n-1}\int_{\frac{\delta}{\varepsilon}}^{\frac{2\delta}{\varepsilon}}\frac{r^{n-1}}{(1+r^{2})^{n}}\ dr\notag\\
&=O(\varepsilon^{n}).\label{Gpsi2}
\end{align}
And to finish, we have
\begin{align}
O\Big(\int_{|x|\leq 2\delta}\int_{\R^{n}}G_{\R^{n}}^{s}(x,y)|\Psi_{\varepsilon}|^{2}(y) |x|^{2}|\Psi_{\varepsilon}|^{2}(x) \ dy dx\Big)&=O\Big(\varepsilon^{2}\int_{0}^{\frac{2\delta}{\varepsilon}}\frac{r^{n+1}}{(1+r^{2})^{n}}\ dr \Big)\notag\\
&=O(\varepsilon^{2}).\notag
\end{align}
Therefore,
$$F_{1}=c_{n}^{\frac{1}{n-1}}a_{0}^{\frac{n}{n-1}}\omega_{n-1}\int_{0}^{+\infty}\frac{r^{n-1}}{(1+r^{2})^{n}}\ dr + O(\varepsilon^{2}).$$
The estimates for $F_{2}$ and $F_{3}$ are relatively simpler. Indeed,
\begin{align}
F_{2}&\leq C \int_{|x|\leq 2\delta}|x|^{3}|\Psi_{\varepsilon}|^{2}\ dx \leq C \varepsilon^{4}\int_{0}^{\frac{2\delta}{\varepsilon}}\frac{r^{n+2}}{(1+r^{2})^{n}}\ dr\notag\\
&\leq  \left\{\begin{array}{ll}
O(\varepsilon^{n-1}) \text{ if } n=3,4\\
O(\varepsilon^{4}|\ln(\varepsilon)| )\text{ if } n=5\\
O(\varepsilon^{4}) \text{ if } n\geq 6
\end{array}
\right..\notag
\end{align}
Similarly,
\begin{align}
F_{3}&\leq C\int_{|x|\leq 2\delta} |x|^{2}|\nabla \Psi_{\varepsilon}| |\Psi_{\varepsilon}|\ dx \leq C \varepsilon^{2}\int_{0}^{\frac{2\delta}{\varepsilon}} \frac{r^{n+1}}{(1+r^{2})^{n}}\ dr\notag\\
&\leq \left\{\begin{array}{ll}
O(\varepsilon^{2}|\ln(\varepsilon)| )\text{ if } n=3\\
O(\varepsilon^{2}) \text{ if } n\geq 4
\end{array}
\right..\notag
\end{align}
Thus,
$$\int_{M}\langle \overline{D} \varphi_{\varepsilon},\varphi_{\varepsilon}\rangle \ dv_{g}\leq c_{n}^{\frac{1}{n-1}}a_{n}^{\frac{n}{n-1}}\omega_{n-1}\int_{0}^{+\infty}\frac{r^{n-1}}{(1+r^{2})^{n}}\ dr + \left\{\begin{array}{ll}
O(\varepsilon^{2}|\ln(\varepsilon)| )\text{ if } n=3\\
O(\varepsilon^{2}) \text{ if } n\geq 4
\end{array}
\right..
$$

\noindent
Now we need to estimate the second term of the energy functional $J_{g}$.
\begin{align}
&\int_{M\times M} |\varphi_{\varepsilon}|^{2}(x)G_{g}^{s}(x,y)|\varphi_{\varepsilon}|^{2}(y)\ dv(x) dv(y)=\int_{|x-y|<\frac{\delta}{2}}|\varphi_{\varepsilon}|^{2}(x)G_{g}^{s}(x,y)|\varphi_{\varepsilon}|^{2}(y)\ dv(x) dv(y)\notag\\
&\quad+\int_{|x-y|>\frac{\delta}{2}}|\varphi_{\varepsilon}|^{2}(x)G_{g}^{s}(x,y)|\varphi_{\varepsilon}|^{2}(y)\ dv(x) dv(y)\notag\\
&=\int_{|x-y|<\frac{\delta}{2}; |x|\leq \frac{\delta}{2}}|\Psi_{\varepsilon}|^{2}(x)[G_{\R^{n}}^{s}(x,y)+r(x,y)]|\Psi_{\varepsilon}|^{2}(y)\ dx dy\notag\\ &\quad +O\Big(\int_{|x|\geq \frac{\delta}{2}}\int_{\R^{n}}G_{\R^{n}}^{s}(x,y)|\Psi_{\varepsilon}|^{2}(y)|x|^{2}|\Psi_{\varepsilon}|^{2}(x)\ dy dx\Big)+O\Big(\Big(\int_{M}|\varphi_{\varepsilon}|^{2}\ dv_{g}\Big)^{2}\Big)\notag\\
&=\int_{|x-y|<\frac{\delta}{2}; |x|\leq \frac{\delta}{2}}|\Psi_{\varepsilon}|^{2}(x)G_{\R^{n}}^{s}(x,y)|\Psi_{\varepsilon}|^{2}(y)\ dx dy+O\Big(\int_{|x-y|<\frac{\delta}{2}}|\Psi_{\varepsilon}|^{2}(x)\frac{1}{|x-y|}|\Psi_{\varepsilon}|^{2}(y)\ dx dy\Big)\notag\\
&\quad+O\Big(\int_{|x|\geq \frac{\delta}{2}}\int_{\R^{n}}G_{\R^{n}}^{s}(x,y)|\Psi_{\varepsilon}|^{2}(y)|x|^{2}|\Psi_{\varepsilon}|^{2}(x)\ dy dx\Big)+ O\Big(\Big(\int_{M}|\varphi_{\varepsilon}|^{2}\ dv_{g}\Big)^{2}\Big)\notag\\
&=\int_{|x|\leq \frac{\delta}{2}}|\Psi_{\varepsilon}|^{2}(x)G_{\R^{n}}^{s}(x,y)|\Psi_{\varepsilon}|^{2}(y)\ dx dy + O\Big(\int_{|x-y|<\frac{\delta}{2};|x|\leq \frac{\delta}{2}}|\Psi_{\varepsilon}|^{2}(x)\frac{1}{|x-y|}|\Psi_{\varepsilon}|^{2}(y)\ dx dy\Big)\notag\\
&\quad +O\Big(\int_{|x|\geq \frac{\delta}{2}}\int_{\R^{n}}G_{\R^{n}}^{s}(x,y)|\Psi_{\varepsilon}|^{2}(y)|x|^{2}|\Psi_{\varepsilon}|^{2}(x)\ dy dx\Big)+O\Big(\Big(\int_{M}|\varphi_{\varepsilon}|^{2}\ dv_{g}\Big)^{2}\Big)\notag
\\ &\quad+O\Big(\Big(\int_{\R^{n}}|\Psi_{\varepsilon}|^{2}\ dx_{g}\Big)^{2}\Big).\notag
\end{align}
Using $(\ref{L2norm})$, we get
$$O\Big(\Big(\int_{M}|\varphi_{\varepsilon}|^{2}\ dv_{g}\Big)^{2}\Big) +O\Big(\Big(\int_{\R^{n}}|\Psi_{\varepsilon}|^{2}\ dx_{g}\Big)^{2}\Big)=O(\varepsilon^{2}).$$
Moreover, from $(\ref{Gpsi})$ and $(\ref{Gpsi2})$, we have
$$
\int_{|x|\leq \frac{\delta}{2}}\int_{\R^{n}}G_{\R^{n}}^{s}(x,y)|\Psi_{\varepsilon}|^{2}(y) |\Psi_{\varepsilon}|^{2}(x) \ dy dx=
c_{n}^{\frac{1}{n-1}}a_{0}^{\frac{n}{n-1}}\omega_{n-1}\int_{0}^{+\infty}\frac{r^{n-1}}{(1+r^{2})^{n}}\ dr +O(\varepsilon^{n}),$$
and
$$O\Big(\int_{|x|\geq \frac{\delta}{2}}\int_{\R^{n}}G_{\R^{n}}^{s}(x,y)|\Psi_{\varepsilon}|^{2}(y)|x|^{2}|\Psi_{\varepsilon}|^{2}(x)\ dy dx\Big)=O(\varepsilon^{2}).$$
It remains to estimate 
$$O\Big(\int_{|x-y|<\frac{\delta}{2}; |x|\leq \frac{\delta}{2}}|\Psi_{\varepsilon}|^{2}(x)\frac{1}{|x-y|}|\Psi_{\varepsilon}|^{2}(y)\ dx dy\Big)$$
$$=O\Big(\int_{|x|\leq \delta }\int_{|y|\leq \delta}|\Psi_{\varepsilon}|^{2}(x)\frac{1}{|x-y|}|\Psi_{\varepsilon}|^{2}(y)\ dx dy\Big)+O(\varepsilon^{2}).$$
Using the Hardy-Littlewood-Sobolev inequality, we have
$$\int_{|x|\leq \delta }\int_{|y|\leq \delta}|\Psi_{\varepsilon}|^{2}(x)\frac{1}{|x-y|}|\Psi_{\varepsilon}|^{2}(y)\ dx dy\leq C\||\Psi_{\varepsilon}|^{2}\|_{L^{\frac{2n}{2n-1}}(B_{\delta})}\leq C\|\Psi_{\varepsilon}\|_{L^{2}}^{2}=O(\varepsilon^{2}).$$
Hence,
$$\int_{M\times M} |\varphi_{\varepsilon}|^{2}(x)G_{s}(x,y)|\varphi_{\varepsilon}|^{2}(y)\ dv(x) dv(y)=c_{n}^{\frac{1}{n-1}}a_{n}^{\frac{n}{n-1}}\omega_{n-1}\int_{0}^{+\infty}\frac{r^{n-1}}{(1+r^{2})^{n}}\ dr +O(\varepsilon^{2}).$$
It follows that 
\begin{align}
J_{g}(\varphi_{\varepsilon})&\leq \frac{1}{4}c_{n}^{\frac{1}{n-1}}a_{n}^{\frac{n}{n-1}}\omega_{n-1}\int_{0}^{+\infty}\frac{r^{n-1}}{(1+r^{2})^{n}}\ dr+O(\varepsilon^{2})\notag\\
&=\overline{Y}+O(\varepsilon^{2}).
\end{align}
\end{proof}

\begin{proof}  (\emph{of Theorem (\ref{thmlambda})})\\
From Lemma \ref{lemma1 exp} and \ref{lemma 2 exp}, we have that
$$J_{\lambda}(\varphi_{\varepsilon})\leq \overline{Y}-\lambda Q(\varepsilon)+O(\varepsilon^{2})$$
and
$$\|\nabla J_{\lambda}(\varphi_{\varepsilon})\|_{H^{*}_{\lambda}}\leq \left\{\begin{array}{ll}
O(\varepsilon|\ln(\varepsilon)|^{\frac{2}{3}} \text{ if } n=3\\
O(\varepsilon) \text{ if } n\geq 4
\end{array}
\right.$$
Therefore, from Proposition \ref{prop exp}, we have for $\varepsilon>0$ and small,
\begin{align}
\delta_{\lambda}&\leq J_{\lambda}(\varphi_{\varepsilon})+O(\|\nabla J_{\lambda}(\varphi_{\varepsilon})\|^{2})\notag\\
&\leq \overline{Y}-\lambda Q(\varepsilon)+\left\{\begin{array}{ll}
O(\varepsilon^{2}|\ln(\varepsilon)|^{\frac{4}{3}}) \text{ if } n=3\\
O(\varepsilon^{2}) \text{ if } n\geq 4
\end{array}
\right. \label{min}\\
&<\overline{Y}.\notag
\end{align}
Since, $J_{\lambda}$ and $\tilde{J}$ satisfy the (PS) condition for energy levels below $\overline{Y}$, we have that $J_{\lambda}$ has a non-trivial critical point $\psi_{\lambda}$.
\end{proof}

\noindent
We finally notice that for $\lambda=0$, $\delta_{0}$ is a conformal invariant of $(M,[g])$ and we will denote it by $\delta_{0}=:\overline{Y}(M,[g])$. With these notations, we see that Corollary \ref{cor1} is a direct consequence of  $(\ref{min})$.

\end{document}